\documentclass[12pt,reqno]{amsart}
\usepackage{bbm}
\usepackage[shortlabels]{enumitem}
\usepackage[hidelinks]{hyperref}
\usepackage[capitalise]{cleveref}
\usepackage{amssymb}
\usepackage{amsthm}
\usepackage{amsmath}
\usepackage{a4wide}
\usepackage{latexsym}
\usepackage{mathrsfs}
\usepackage[T1]{fontenc}
\usepackage[latin1]{inputenc}
\usepackage[v2,tips]{xy}

\newtheorem{theorem}{Theorem}[section]
\newtheorem{lemma}[theorem]{Lemma}
\newtheorem{proposition}[theorem]{Proposition}
\newtheorem{corollary}[theorem]{Corollary}

\theoremstyle{definition}
\newtheorem{definition}[theorem]{Definition}
\newtheorem{notation}[theorem]{Notation}

\newtheorem{example}[theorem]{Example}
\newtheorem{remark}[theorem]{Remark}

\numberwithin{equation}{section}
\newcommand{\Proj}{\operatorname{Proj}}
\newcommand{\N}{\mathbb{N}}

\newcommand{\K}{\mathbb{K}}

\newcommand{\viii}[1]{{\left\vert\kern-0.25ex\left\vert\kern-0.25ex\left\vert #1   \right\vert\kern-0.25ex\right\vert\kern-0.25ex\right\vert}}
\newcommand{\one}{\mathbbm{1}}
\newcommand{\A}{\mathcal{A}}
\newcommand{\mcb}{\mathscr{B}}
\newcommand{\mcg}{\mathscr{G}}
\newcommand{\mck}{\mathscr{K}}
\newcommand{\mci}{\mathscr{I}}
\newcommand{\mcj}{\mathscr{J}}

\makeatletter
\newcommand\notni{\mathrel{\m@th\mathpalette\canc@l\owns}}
\newcommand\canc@l[2]{{\ooalign{$\hfil#1/\mkern1mu\hfil$\crcr$#1#2$}}}
\makeatother

\newcommand{\smashw}[2][l]{{\text{\makebox[0pt][#1]{$#2$}}}}

\renewcommand{\leq}{\ensuremath{\leqslant}}
\renewcommand{\le}{\ensuremath{\leqslant}}
\renewcommand{\geq}{\ensuremath{\geqslant}}
\renewcommand{\ge}{\ensuremath{\geqslant}}

\renewcommand{\epsilon}{\ensuremath{\varepsilon}}
\renewcommand{\phi}{\ensuremath{\varphi}}


\title[Uniqueness of algebra norm on quotients of $\mathscr{B}(X)$]{Uniqueness of algebra norm on quotients of the  algebra of bounded operators on a Banach space}

\dedicatory{In memoriam: H.~G.~Dales (1944--2022)}

\author[M.~Arnott]{Max Arnott}
\address[M.~Arnott]{Department of Mathematics and Statistics, Fylde College, Lancaster University, Lancaster, LA1 4YF, United Kingdom}
\email{m.arnott@lancaster.ac.uk}

\author[N.~J.~Laustsen]{Niels Jakob Laustsen}
\address[N.~J.~Laustsen]{Department of Mathematics and Statistics, Fylde College, Lancaster University, Lancaster, LA1 4YF, United Kingdom}
\email{n.laustsen@lancaster.ac.uk}

\begin{document}
\begin{abstract}
We show that for each of the following Banach spaces~$X$,  the quotient algebra $\mcb(X)/\mci$ has a unique algebra norm for every closed ideal $\mci$ of $\mcb(X)\colon$ 
\begin{itemize}   
    \item $X= \bigl(\bigoplus_{n\in\N}\ell_2^n\bigr)_{c_0}$\quad and its dual,\quad $X=  \bigl(\bigoplus_{n\in\N}\ell_2^n\bigr)_{\ell_1}$,
    \item $X=  \bigl(\bigoplus_{n\in\N}\ell_2^n\bigr)_{c_0}\oplus c_0(\Gamma)$\quad and its dual, \quad $X= \bigl(\bigoplus_{n\in\N}\ell_2^n\bigr)_{\ell_1}\oplus\ell_1(\Gamma)$,\quad for an uncountable cardinal number~$\Gamma$,
    \item $X = C_0(K_{\mathcal{A}})$, the Banach space of continuous functions vanishing at infinity on the locally compact Mr\'{o}wka space~$K_{\mathcal{A}}$ induced by an uncountable, almost disjoint family~$\mathcal{A}$  of infinite subsets of~$\mathbb{N}$, constructed such that $C_0(K_{\mathcal{A}})$ admits ``few operators''.
\end{itemize}
 Equivalently, this result states that every homomorphism from~$\mathscr{B}(X)$ into a Banach alge\-bra is continuous and has closed range. The key step in our proof is to show that the identity operator on a suitably chosen Banach space factors through every operator in $\mathscr{B}(X)\setminus\mci$  with control over the norms of the operators used in the factorization. These quantitative factorization results may be of independent interest. 
\end{abstract}

\maketitle

\section{Introduction and statement of main results}

\noindent By an \emph{algebra norm} on  a (real or complex) algebra~$\mathscr{A}$, we understand
a norm~$\lVert\,\cdot\,\rVert$ on~$\mathscr{A}$ that is submultiplicative in the sense that $\lVert ab\rVert\le \lVert a\rVert\,\lVert b\rVert$ for every $a,b\in\mathscr{A}$. A~\emph{normed algebra} is an algebra equipped with an algebra norm. We say that a normed algebra $\mathscr{A}$ has a \textit{unique algebra norm} if every algebra norm on~$\mathscr{A}$ is equivalent to the given norm.

In the case where~$\mathscr{A}$ is a Banach algebra (which will be the case for us), this notion should not be confused with the weaker notion of having a unique \emph{complete} algebra norm; see  \Cref{R:ucan} for a more detailed discussion of the difference between  these two notions.

We shall focus  on  the Banach algebra~$\mathscr{B}(X)$ of bounded operators on a Banach space~$X$ and its quotients by closed ideals. One of the earliest results about
this Banach algebra (predating even the term ``Banach algebra'') is due to Eidel\-heit~\cite[Theorem~1]{E}, who showed that~$\mathscr{B}(X)$ has a unique complete algebra norm for every Banach space~$X$. In fact, building on Eidelheit's methods, one can deduce a stronger  conclusion, which we shall state in \Cref{thm:yood} below, once we  have established the necessary terminology.

Many related results have subsequently been obtained; \cite[\S{}0]{Ty2} contains an excellent survey of what was known around the turn of the millenium. We would like to highlight the following results, which provide the most important context for our work:
\begin{enumerate}[label={\normalfont{(\roman*)}}]
\item\label{meyer} Meyer~\cite{M} proved that the Calkin algebra $\mathscr{B}(X)/\mathscr{K}(X)$ has a unique algebra norm for each of the classical sequence spaces $X=c_0$ and $X=\ell_p$, where $1\le p<\infty$. For later reference, we recall the classical result of  Gohberg, Markus and Feldman~\cite{GMF} that the ideal~$\mathscr{K}(X)$ of compact operators is the unique non-trivial closed ideal of~$\mathscr{B}(X)$ for each of these Banach spaces~$X$. 
\item\label{Ware} In his dissertation~\cite{GW}, Ware launched a systematic attack on the uniqueness-of-algebra-norm question for Calkin algebras,  generalizing the aforementioned results of Meyer by showing that the Calkin algebra~$\mathscr{B}(X)/\mathscr{K}(X)$  has a unique algebra norm for a wide range of Banach spaces~$X$, including the following \cite[\S\S{}5.2--5.5]{GW}:
  \begin{itemize}
  \item every finite direct sum of spaces from the family $\{c_0\}\cup\{\ell_p : 1\le p<\infty\};$
  \item the infinite direct sums $\bigl(\bigoplus_{n\in\mathbb{N}}\ell_p^n\bigr)_{c_0}$ and $\bigl(\bigoplus_{n\in\mathbb{N}}\ell_p^n\bigr)_{\ell_q}$ for $1\le p\le\infty$ and $1\le q<\infty;$
  \item the Tsirelson space~$T$ and its dual~$T^*$;
  \item the quasi-reflexive James spaces $J_p$ for $1<p<\infty$.  
  \end{itemize}
\item\label{GW_nonsep}  In a somewhat different direction, Ware \cite[Section~6]{GW} generalized Meyer's results to the non-separable setting as follows. Take an  uncountable index set~$\Gamma$, and set  $X=c_0(\Gamma)$ or $X=\ell_p(\Gamma)$ for some $1\le p<\infty$. Then the quotient $\mathscr{B}(X)/\mathscr{I}$ has a unique algebra norm for every closed ideal~$\mathscr{I}$  of~$\mathscr{B}(X)$. Daws' classification~\cite{DAWS} of the closed ideals of~$\mathscr{B}(X)$ for these Banach spaces~$X$ plays a key role in this work. 
\item\label{JPS} More recently, Johnson and Phillips~\cite{J-IHABB} have answered one of the main questions left open in Ware's dissertation~\cite{GW} by showing that the Calkin al\-ge\-bra $\mathscr{B}(L_p[0,1])/\mathscr{K}(L_p[0,1])$ has a unique algebra norm whenever $1<p<\infty$. Further\-more, for $p\ne2$, they have shown that $\mathscr{B}(L_p[0,1])$ contains a closed ideal~$\mathscr{I}$ for which the quotient  $\mathscr{B}(L_p[0,1])/\mathscr{I}$ admits at least two inequivalent algebra norms.
\end{enumerate}  

Motivated by these results, especially~\ref{meyer}, \ref{GW_nonsep} and~\ref{JPS}, we asked ourselves for which  Banach spaces~$X$ other than~$c_0(\Gamma)$ and $\ell_p(\Gamma)$ for $1\le p<\infty$ is it true that every quotient of~$\mathscr{B}(X)$  by a closed ideal has a unique algebra norm? Our focus has been on the (relatively few) Banach spaces~$X$ for which all of the closed ideals of~$\mathscr{B}(X)$ have been classified. The following theorem summarizes our main findings. 

\begin{theorem}\label{mainthm}
Every quotient of $\mcb(X)$ by one of its closed ideals has a unique algebra norm for each of the following Banach spaces~$X\colon$ 
\begin{enumerate}[label={\normalfont{(\roman*)}}]
\item\label{mainthm1} $X=\bigl(\bigoplus_{n\in\mathbb{N}}\ell_2^n\bigr)_{c_0}$\quad and\quad  $X=\bigl(\bigoplus_{n\in\mathbb{N}}\ell_2^n\bigr)_{\ell_1};$
\item\label{mainthm1.5}  $X = \bigl(\bigoplus_{n\in\mathbb{N}}\ell_2^n\bigr)_{c_0}\oplus c_0(\Gamma)$\quad and\quad $X = \bigl(\bigoplus_{n\in\mathbb{N}}\ell_2^n\bigr)_{\ell_1}\oplus\ell_1(\Gamma)$ \quad for an uncountable in\-dex set~$\Gamma;$ 
\item\label{mainthm2} $X = C_0(K_{\mathcal{A}})$, the Banach space of continuous functions vanishing at infinity on the Mr\'{o}wka space~$K_{\mathcal{A}}$ induced by an uncountable, almost disjoint family~$\mathcal{A}$ of infinite subsets of~$\mathbb{N}$, chosen such that~$C_0(K_{\mathcal{A}})$ admits ``few operators''. 
\end{enumerate}   
\end{theorem}

\noindent 
We refer to \Cref{c0KA} for details of the terminology used in~\ref{mainthm2}. Note that~\ref{mainthm1} can be viewed as a special case of~\ref{mainthm1.5} by allowing the index set~$\Gamma$ to be countable  because \[ \Bigl(\bigoplus_{n\in\mathbb{N}}\ell_2^n\Bigr)_{c_0}\cong\Bigl(\bigoplus_{n\in\mathbb{N}}\ell_2^n\Bigr)_{c_0}\oplus c_0\qquad \text{and}\qquad \Bigl(\bigoplus_{n\in\mathbb{N}}\ell_2^n\Bigr)_{\ell_1}\cong \Bigl(\bigoplus_{n\in\mathbb{N}}\ell_2^n\Bigr)_{\ell_1}\oplus\ell_1. \]

This manuscript is organized as follows. In \Cref{S:prelim}, we introduce some background material and associated terminology, notably concerning maximality, minimality and incompressibility of algebra norms, and we explain the roles that automatic continuity of homomorphisms  and quantitative factorizations of idempotents play in their study. We also give a ``skeleton proof'' of \Cref{mainthm} and  some illustrative examples for context.

\Cref{Section3} contains a quantitative factorization theorem for the identity operator on~$c_0$ that we shall use in the proofs of  \Cref{mainthm}\ref{mainthm1.5} and~\ref{mainthm2}. It is based on theorems of Rosen\-thal and  Dow\-ling--Ran\-drianan\-toa\-nina--Turett.

With these preparations under our belt, we complete the proofs of the three parts of \Cref{mainthm} in Sec\-tions~\ref{c0sum}--\ref{c0KA}. We believe that our approach via  quantitative factorizations of the identity operators on suitably chosen Banach spaces 
may be of interest to researchers studying topics quite different from uniqueness of algebra norm and refer to Theorems~\ref{T:LLR_LSZ}, \ref{P:IFPlargeideals}, \ref{finalCKresult} and \Cref{calkindirectsumincompressible} 
for our main conclusions stated in terms of such factorizations.

Finally, in \Cref{section5} we show that the quantitative factorization route we took to proving \Cref{finalCKresult} is not only sufficient, but also necessary.

\section{Preliminaries, including a skeleton proof of \Cref{mainthm}}\label{S:prelim}
\noindent
All normed spaces and algebras are over  the same scalar field~$\K$,  either the real or the complex numbers. 
As usual, for a locally compact Hausdorff space~$K$, $C_0(K)$ denotes the Banach space of continuous functions $f\colon K\to\K$  which ``vanish at infinity'' in the sense that the set $\{ k\in K: \lvert f(k)\rvert\le\epsilon\}$ is compact for every $\epsilon>0$. 

The term ``operator'' means a bounded linear map between normed spaces.
We write~$B_X$ for  the closed unit ball of a normed space~$X$, while~$I_X$ (or simply~$I$ when~$X$ is clear from the context) denotes the identity operator on~$X$. For a Banach space with a basis and $n\in\N$, $P_n$ will usually denote the $n^{\text{th}}$ basis projection. Other notation will be introduced as and when required.  

Equivalence of norms plays a central role in this investigation. The definition splits naturally in two parts, an upper and a lower bound. This is the motivation behind the following terminology.

\begin{definition}
  Let \mbox{$(\mathscr{A},\lVert\,\cdot\,\rVert)$} be a normed algebra. We say that the given norm  $\lVert\,\cdot\,\rVert$ on~$\mathscr{A}$ is \textit{maximal} if, for every algebra norm~$\viii{\,\cdot\,}$ on $\mathscr{A}$, there is a constant $C_1>0$ such that $\viii{a}\leq C_1\|a\|$ for every $a \in \mathscr{A}$.   Analogously,  we say that
  \mbox{$\lVert\,\cdot\,\rVert$} is \textit{minimal} if, for every algebra norm~$\viii{\,\cdot\,}$ on $\mathscr{A}$,  there is a constant $C_2>0$ such that $\|a\|\leq C_2\viii{a}$ for every $a \in \mathscr{A}$.
  We say that $\mathscr{A}$ has a \textit{unique algebra norm} if the given norm is both maximal and minimal, or in other words if every algebra norm on~$\mathscr{A}$ is  equivalent to the given norm.
\end{definition} 

\begin{remark}  One usually requires that an algebra norm on  a unital algebra~$\mathscr{A}$ must take the value~$1$ at the multiplicative identity~$1_\mathscr{A}$. We shall not adhere to this convention. This will not cause any problems because, given any algebra norm  $\lVert\,\cdot\,\rVert$ on~$\mathscr{A}$,
\[ \viii{a} = \sup\bigl\{\lVert ab\rVert :  b\in\mathscr{A},\,\lVert b\rVert\le 1\bigr\}\qquad (a\in\mathscr{A}) \] 
defines an equivalent algebra norm on~$\mathscr{A}$ which satisfies $\viii{1_\mathscr{A}} =1$. 
\end{remark}

\begin{remark}\label{R:ucan} For a Banach algebra~$\mathscr{A}$, there is a related notion of having a \textit{unique complete algebra norm,} which asserts that every complete algebra norm on~$\mathscr{A}$ is equivalent to the given norm. The Banach Isomorphism Theorem implies that whenever the given norm on a Banach algebra~$\mathscr{A}$ is either minimal or maximal, then it is automatically equivalent to any other complete norm on~$\mathscr{A}$, so~$\mathscr{A}$ has unique complete algebra norm. 
\end{remark}   

It is well-known that maximality of the norm can be rephrased in terms of automatic continuity of homomorphisms, which is one of the oldest topics in the theory of Banach algebras. Here, and elsewhere,   the term ``homomorphism''  means a linear and multiplicative map between two algebras. There are many ways to express this relationship. We have chosen the following, which reflects the applications we have in mind. A much more comprehensive result can be found in \cite[Theorem~6.1.5]{pal}.

\begin{proposition}\label{maximalnorm} The following three conditions are equivalent for a normed algebra $(\mathscr{A},\lVert\,\cdot\,\rVert)\colon$ 
\begin{enumerate}[label={\normalfont{(\alph*)}}]
\item\label{maximalnorm1} The given norm  $\lVert\,\cdot\,\rVert$ on~$\mathscr{A}$ is maximal.
\item\label{maximalnorm2} Every homomorphism from $\mathscr{A}$ into a normed algebra is continuous.
\item\label{maximalnorm3} Every injective homomorphism from~$\mathscr{A}$ into a Banach algebra is continuous.
\end{enumerate}
Now suppose that $(\mathscr{A},\lVert\,\cdot\,\rVert)$ is a Banach algebra. Then conditions~\ref{maximalnorm1}--\ref{maximalnorm3} are equivalent to the following two conditions: 
\begin{enumerate}[label={\normalfont{(\alph*)}}]
\setcounter{enumi}{3}
\item\label{maximalnorm5} The quotient norm on~$\mathscr{A}/\mathscr{I}$ is maximal for every closed ideal~$\mathscr{I}$ of $\mathscr{A}$.
\item\label{maximalnorm4} Every homomorphism from $\mathscr{A}/\mathscr{I}$ into a normed algebra is continuous for every closed ideal~$\mathscr{I}$ of~$\mathscr{A}$.
\end{enumerate}
\end{proposition}

\begin{proof} This is standard. For example,   it is easy to adapt the proof of  \cite[Proposition~2.1.7]{dales} to verify that conditions~\ref{maximalnorm1}--\ref{maximalnorm3} are equivalent, and therefore conditions~\ref{maximalnorm5} and~\ref{maximalnorm4} are also equivalent. This part of the proof does not require completeness of~$\mathscr{A}$. Completeness (in the shape of the Open Mapping Theorem) is, however,  required to prove that~\ref{maximalnorm2} implies~\ref{maximalnorm4}; see \cite[Proposition~2.1.5]{dales} for details. Finally, the implication \ref{maximalnorm4}$\Rightarrow$\ref{maximalnorm2} is  trivial.  
\end{proof}

B.~E.~Johnson~\cite{johnson} proved what has become the classical automatic continuity result  for homo\-mor\-phisms from the Banach algebra~$\mathscr{B}(X)$. We state it in a simplified non-technical form, as it will suffice for our purposes.  
\begin{theorem}[Johnson]\label{T:BEJ} 
 Let $X$ be a Banach space which is isomorphic to its square~$X\oplus X$. Then every homomorphism from~$\mathscr{B}(X)$ into a Banach algebra is continuous. 
\end{theorem}  

\Cref{maximalnorm} has  a counterpart  for minimality. Its statement involves the following notion: a linear map $T\colon X\to Y$ between normed spaces $X$ and $Y$ is \emph{bounded below}  if there is a constant $\eta>0$ such that $\lVert Tx\rVert\ge \eta\lVert x\rVert$ for every $x\in\ X$. If we wish to specify the constant, we say that~$T$ is \emph{bounded below by}~$\eta$.

\begin{lemma}\label{prop:minimality}
Let $(\mathscr{A},\lVert\,\cdot\,\rVert)$ be a normed algebra. Then its norm  $\lVert\,\cdot\,\rVert$ is minimal if and only if every injective homomorphism from~$\mathscr{A}$ into a normed algebra is bounded below.
\end{lemma}  

\begin{proof} Straightforward; we refer to the proof of \cite[Theorem~6.1.5(a)--(b)]{pal} for details. \end{proof}

The following result originates in the work of Eidelheit~\cite[Lemma~1 and Theorem~1]{E}, as already mentioned in the Introduction. It can for instance be found in \cite[Theorem~5.1.14]{dales} or \cite[Theorem, page 107]{pal}.

\begin{theorem}\label{thm:yood}
  Let $X$ be a Banach space, and let~$\mathscr{A}$ be a subalgebra of~$\mathscr{B}(X)$ which contains the ideal of finite-rank operators. Then the operator norm on~$\mathscr{A}$ is minimal.   
\end{theorem}

In a forthcoming paper~\cite{J-IHABB}, W.~B.~Johnson and N.~C.~Phillips introduce and explore the following two notions. 

\begin{definition}
  A normed algebra $\mathscr{A}$ is:
\begin{itemize}
\item  \textit{incompressible} if every continuous, injective homo\-mor\-phism from~$\mathscr{A}$ into a normed algebra is bounded below;
\item \textit{uni\-form\-ly in\-com\-pressible} if there is a function $f\colon(0,\infty)\to (0,\infty)$ such that every continuous, injective homomorphism~$\varphi$ from~$\mathscr{A}$ into a  normed alge\-bra is bounded below by~$f(\lVert\varphi\rVert)$.
\end{itemize}
\end{definition}

\Cref{prop:minimality} shows that minimality of the norm implies incompressibility and that the two notions are equivalent if every homomorphism from~$\mathscr{A}$ into a normed algebra is continuous. More precisely, we have  the following result. 
\begin{lemma}\label{Cor:uniquealgnorm}
  A normed algebra $(\mathscr{A},\lVert\,\cdot\,\rVert)$ has a unique algebra norm if and only if~$\mathscr{A}$ is in\-com\-pressible and the given norm~$\lVert\,\cdot\,\rVert$ is maximal. 
\end{lemma}

\begin{proof}
  This is immediate from \Cref{prop:minimality} because, by \Cref{maximalnorm},  maximality of the norm~$\lVert\,\cdot\,\rVert$ means that every homomorphism from~$\mathscr{A}$ into a normed algebra is continuous. 
\end{proof}

\begin{corollary}\label{C:uniquenormQuot} The following three conditions are equivalent for a Banach algebra~$\mathscr{A}\colon$
  \begin{enumerate}[label={\normalfont{(\alph*)}}]
    \item\label{C:uniquenormQuot1}  Every quotient algebra of~$\mathscr{A}$ by a closed ideal has a unique algebra norm. 
    \item\label{C:uniquenormQuot2}  Every homomorphism from~$\mathscr{A}$ into a Banach algebra is continuous and has closed range.  
    \item\label{C:uniquenormQuot3}   Every homomorphism from~$\mathscr{A}$ into a Banach algebra is continuous, and~$\mathscr{A}/\mathscr{I}$ is incompressible for every  closed ideal~$\mathscr{I}$ of~$\mathscr{A}$.  
\end{enumerate}      
\end{corollary}

\begin{proof}
\ref{C:uniquenormQuot1}$\Rightarrow$\ref{C:uniquenormQuot2}. Let  $\varphi\colon\mathscr{A}\to\mathscr{B}$ be a  homomorphism into a Banach algebra~$\mathscr{B}$. 
 \Cref{maximalnorm}  shows that~$\varphi$ is continuous because   the norm on~$\mathscr{A}$ is maximal.
  In particular, the ideal $\ker\varphi$ is closed, and the quotient norm on $\mathscr{A}/\ker\varphi$ is minimal by hypothesis.
  Let  $\widetilde{\varphi}\colon \mathscr{A}/\ker\varphi\to\mathscr{B}$ be  the induced  homomorphism given by $\widetilde{\varphi}(a+\ker\varphi) = \varphi(a)$; it satisfies:
 \begin{enumerate}[label={\normalfont{(\roman*)}}] 
\item\label{C:uniquenormQuot:i} $\widetilde{\varphi}[\mathscr{A}/\ker\varphi] =   \varphi[\mathscr{A}]$,
\item\label{C:uniquenormQuot:ii} $\widetilde{\varphi}$  is  injective and therefore  bounded below by \Cref{prop:minimality},
  \item  $\widetilde{\varphi}$  is continuous, hence it has closed range by~\ref{C:uniquenormQuot:ii}, and therefore so does $\varphi$ by~\ref{C:uniquenormQuot:i}.
\end{enumerate}

 \ref{C:uniquenormQuot2}$\Rightarrow$\ref{C:uniquenormQuot3}. Let $\mathscr{I}$ be a closed ideal of $\mathscr{A}$, and consider a continuous, injective  homo\-mor\-phism $\varphi\colon\mathscr{A}/\mathscr{I}\to\mathscr{B}$ into a normed algebra~$\mathscr{B}$. By hypothesis, the composite homomorphism $\iota\varphi\pi\colon\mathscr{A}\to\widehat{\mathscr{B}}$ has closed range, where $\pi\colon\mathscr{A}\to\mathscr{A}/\mathscr{I}$ is  the quotient map and $\iota\colon\mathscr{B}\to\widehat{\mathscr{B}}$ is the isometric embedding of~$\mathscr{B}$ into its completion~$\widehat{\mathscr{B}}$. It follows that the injection~$\iota\varphi$ has closed range because~$\pi$ is surjective, so~$\iota\varphi$ is  bounded below, and  therefore the  same is true for~$\varphi$.

 \ref{C:uniquenormQuot3}$\Rightarrow$\ref{C:uniquenormQuot1}. This follows from \Cref{maximalnorm} and \Cref{Cor:uniquealgnorm}.
\end{proof}    

We can now outline how we shall prove \Cref{mainthm}.
\begin{proof}[Skeleton proof of Theorem~{\normalfont{\ref{mainthm}}}] 
According to \Cref{C:uniquenormQuot}, we must show that each of the Banach spaces~$X$ listed in clauses~\ref{mainthm1}--\ref{mainthm2} of \Cref{mainthm} satisfies: 
 \begin{enumerate}[label={\normalfont{(\arabic*)}}] 
 \item\label{Step1}  every homomorphism from~$\mathscr{B}(X)$ into a Banach algebra is continuous;
 \item\label{Step2} $\mathscr{B}(X)/\mathscr{I}$ is incompressible for every  closed ideal~$\mathscr{I}$ of~$\mathscr{B}(X)$.
 \end{enumerate}
 In each case, \ref{Step1} is already known to hold true.  This follows from Johnson's result stated in \Cref{T:BEJ} above for the Banach spaces listed in clauses~\ref{mainthm1}--\ref{mainthm1.5} because they are all isomorphic to their squares, while \cite[Corollary~39]{KL} verifies it for the particular Banach space $X=C_0(K_{\mathcal{A}})$ considered in clause~\ref{mainthm2}, which fails to be isomorphic to its square.
 
Considering~\ref{Step2}, we may clearly suppose that the ideal~$\mathscr{I}$ is proper, and in view of \Cref{thm:yood},  we may also suppose that~$\mathscr{I}$ is non-zero. Hence, it remains to show that
 \begin{equation}\label{Proof:todo} 
 \mathscr{B}(X)/\mathscr{I}\ \text{is incompressible for every  non-trivial closed ideal}\ \mathscr{I}\ \text{of}\ \mathscr{B}(X),
 \end{equation}
which we shall  do in Sections~\ref{c0sum}--\ref{c0KA}. 
\end{proof} 

When verifying~\eqref{Proof:todo}, we shall follow the lead of Johnson and Phillips~\cite{J-IHABB} and use ``quan\-ti\-ta\-tive factorizations'' of idempotent operators to deduce  that the quotient algebra $\mathscr{B}(X)/\mathscr{I}$ is uniformly incompressible. This is a natural tool to bring to bear on the problem at hand because factorizations of idempotent operators already play a key role in the classifications of the closed ideals of~$\mathscr{B}(X)$ for each of the Banach spaces~$X$ we consider. The information we need to add to these results is that the factorizations can be carried out  with uniform norm bounds on the auxiliary operators. We believe that the results we obtain in this regard will be of interest to researchers working on problems very different from uniqueness of algebra norm. 
The following notion is at the heart of our approach.

\begin{definition}\label{D:IFP} Let  $C\ge 1$.
   A normed algebra~$\mathscr{A}$ has the \emph{idempotent factorization property with constant}~$C$ (abbreviated $C$-IFP) if, for every $a\in\mathscr{A}$ of norm~$1$, there are $b,c \in \mathscr{A}$ with $\lVert b\rVert\,\lVert c\rVert\leq C$  such that $bac$ is a non-zero idempotent.
\end{definition}

It is sometimes more convenient to use the following equivalent formulation of this definition: $\mathscr{A}$ has the $C$-IFP if, for every $a\in\mathscr{A}\setminus\{0\}$, there are $b,c \in \mathscr{A}$ with $\lVert b\rVert\,\lVert c\rVert\leq C/\lVert a\rVert$  such that $bac$ is a non-zero idempotent.

Using this terminology, we can express the basic observation of Johnson and Phillips \cite{J-IHABB} as follows; see also~\cite{Me2} for an earlier related, more abstract result. We include a short proof for ease of reference.

\begin{proposition}\label{P:JPSlemma0.2}
    Every normed algebra which has the $C$-IFP for some constant $C\ge 1$ is  uniformly incompressible.
\end{proposition}

\begin{proof} Suppose that $\mathscr{A}$ is a normed algebra with the $C$-IFP for some $C\ge1$, and let $\varphi\colon\mathscr{A}\to\mathscr{B}$ be a continuous, injective homomorphism into some normed algebra~$\mathscr{B}$. For every $a\in\mathscr{A}\setminus\{0\}$, we can find $b,c\in\mathscr{A}$ with $\lVert b\rVert\,\lVert c\rVert\leq C/\lVert a\rVert$  such that $p=bac$ is a non-zero idempotent. Then $\varphi(p)$ is a non-zero idempotent in~$\mathscr{B}$, so
\[ 1\le \lVert\varphi(p)\rVert\le \lVert \varphi(b)\rVert\,\lVert \varphi(a)\rVert\,\lVert \varphi(c)\rVert\leq \lVert\varphi\rVert^2\,\frac{C}{\lVert a\rVert}\,\lVert \varphi(a)\rVert. \] 
Rearranging this inequality, we see that~$\varphi$ is bounded below by $1/(C\lVert\varphi\rVert^2)$, and consequently 
the definition of  uniform incompressibility is satisfied with respect to the function $f\colon(0,\infty)\to(0,\infty)$ given by $f(t) = 1/(Ct^2)$. 
\end{proof}

We shall usually apply a variant of \Cref{P:JPSlemma0.2}, adapted to the context of operator ideals. According to Pietsch~\cite{pie}, an \textit{operator ideal} is an assignment $\mathscr{J}$ which designates to every pair~$(X,Y)$ of Banach spaces a subspace~$\mathscr{J}(X,Y)$ of~$\mcb(X,Y)$ satisfying:
\begin{itemize}
    \item $\mathscr{J}(X,Y) \neq \{0\}$ for some Banach spaces $X$ and $Y$;
    \item $UTS\in\mathscr{J}(W,Z)$ whenever $W$, $X$, $Y$ and~$Z$ are Banach spaces and $S\in \mcb(W,X)$, $T \in \mathscr{J}(X,Y)$ and $U \in \mcb(Y,Z)$ are operators. 
\end{itemize}
An operator ideal~$\mathscr{J}$ is \emph{closed} if  the subspace~$\mathscr{J}(X,Y)$ is closed in the operator norm for every pair of Banach spaces~$(X,Y)$. As usual, we abbreviate~$\mathscr{J}(X,X)$ to~$\mathscr{J}(X)$. 

\begin{lemma}\label{Factorisation}
  Let $X$ be a Banach space and~$\mathscr{J}$ a closed operator ideal. Suppose that there exists a constant $C\ge 1$ such that, for every $T\in\mcb(X)\setminus\mathscr{J}(X)$ with $\lVert T+\mathscr{J}(X)\rVert = 1$, there exist a Banach space~$Y$ and operators $R\in \mcb(X,Y)$ and $S\in \mcb(Y,X)$ with  $\lVert R\rVert\,\lVert S\rVert\leq C$ such that  $RTS+\mathscr{J}(Y)$ is a non-zero idempotent in~$\mcb(Y)/\mathscr{J}(Y)$. Then $\mcb(X)/\mathscr{J}(X)$ has the $C^2$-IFP and is therefore uniformly incompressible.
\end{lemma}

\begin{proof} If $\mathscr{J}(X)=\mcb(X)$, there is nothing to prove, so suppose that the closed ideal~$\mathscr{J}(X)$ is proper, and choose $C\ge 1$ as specified. Then, for every $t=T+\mathscr{J}(X)\in \mcb(X)/\mathscr{J}(X)$ with $\lVert t\rVert = 1$, we can find a Banach space~$Y$ and operators $R\in \mcb(X,Y)$ and $S\in \mcb(Y,X)$ with $\lVert R\rVert\,\lVert S\rVert\leq C$ such that $P+\mathscr{J}(Y)$ is a non-zero idempotent, where  $P=RTS\in\mcb(Y)$.  Then $P\notin\mathscr{J}(Y)$,  and we have   
$P^3-P = (I_Y+P)(P^2-P)\in\mathscr{J}(Y)$,
so the operator $Q= (TSR)^2\in\mcb(X)$ satisfies
\[ RQTS = P^3 \notin\mathscr{J}(Y)\qquad\text{and}\qquad Q^2 - Q = TS(P^3-P)R\in\mathscr{J}(X). \] 
It follows that  $q=Q+\mathscr{J}(X)$ is a non-zero idempotent in~$\mcb(X)/\mathscr{J}(X)$. Therefore its image under the  injective homomorphism~$\varphi$ is also a non-zero idempotent, which implies that $\lVert\varphi(q)\rVert\ge 1$. Set $u=SR+\mathscr{J}(X)\in\mcb(X)/\mathscr{J}(X)$, and observe that \mbox{$\lVert u\rVert\le \lVert S\rVert\,\lVert R\rVert\le C$} and $q=(tu)^2$. Combining these facts, we conclude that~$\mcb(X)/\mathscr{J}(X)$ has the $C^2$-IFP.
\end{proof}

\begin{remark}\label{R:factoringID}
    In our applications of \Cref{Factorisation}, a stronger hypothesis will be satisfied: we can always take the Banach space~$Y$ and the operators~$R$ and~$S$ with  $\lVert R\rVert\,\lVert S\rVert\leq C$ such that the non-zero idempotent $RTS+\mathscr{J}(Y)$ is equal to $I_Y+\mathscr{J}(Y)$.  In this case we can draw the stronger conclusion that $\mcb(X)/\mathscr{J}(X)$ has the $C$-IFP. The reason is that, in the notation already introduced in the proof, the  hypothesis that $I_Y - RTS\in \mathscr{J}(Y)$ implies that 
    \[ q = (tu)^2 = TSRTSR + \mathscr{J}(X) = TSI_YR + \mathscr{J}(X) = tu, \]  
   and $\lVert u\rVert\le C$ as before.
\end{remark}

Suppose that the operator ideal~$\mathscr{J}$ is the ideal of compact operators, so that we consider the Calkin algebra $\mathscr{B}(X)/\mathscr{K}(X)$ of~$X$. A result of Barnes \cite[Lemma~1, including the remark following it]{BB} implies that idempotents lift from the Calkin algebra to~$\mathscr{B}(X)$. Our next lemma will allow us to establish quantitative versions of this result in certain cases. More precisely, under the hypothesis of \Cref{R:factoringID} that we can factor $I_Y+\mathscr{K}(Y)$ through $T+\mathscr{K}(X)$ with control over the norms of the factors, it turns out that we can lift the factorization to obtain an analogous quantitative factorization for the operator~$T$ itself, provided that the Banach space~$Y$ is sufficiently ``nice'' (which it will be in our applications, where $Y=c_0$ or $Y=\ell_1$; see also \Cref{R:perturb} for how to handle some other cases). As well as formally strengthening our conclusions, we believe that presenting them in this way will make them more accessible to researchers from other fields. 
% Achieving this relies on the following simple result.  

\begin{lemma}\label{L:perturb}
Let $Y$ be a Banach space with a normalized, bimonotone basis~$(b_n)_{n\in\N}$, and suppose that 
\begin{equation}\label{L:perturb:eq1} \biggl\lVert\sum_{n=1}^m \alpha_n b_n\biggr\rVert = \biggl\lVert\sum_{n=1}^m \alpha_n b_{n+1}\biggr\rVert\qquad (m\in\N,\, \alpha_1,\ldots,\alpha_m\in\K). \end{equation}
Then, for every $C>1$ and every operator $T\in\mathscr{B}(Y)$ such that $I_Y-T$ is compact, 
there are operators $U,V\in\mathscr{B}(Y)$ with $\lVert U\rVert\,\lVert V\rVert\le C$ such that $UTV = I_Y$. 
\end{lemma}

\begin{proof} Equation~\eqref{L:perturb:eq1} shows that the linear map~$R$ given by $Rb_n = b_{n+1}$ for $n\in\N$ is an isometry. Now consider the linear map~$L$ given by $Lb_1=0$ and $Lb_{n+1}=b_n$ for $n\in\N$. It satisfies 
  \begin{align*} \biggl\lVert L\Bigl(\sum_{n=1}^{m+1} \alpha_n b_n\Bigr)\biggr\rVert = \biggl\lVert\sum_{n=1}^m \alpha_{n+1} b_n\biggr\rVert
   &=  \biggl\lVert\sum_{n=1}^m \alpha_{n+1} b_{n+1}\biggr\rVert = \biggl\lVert(I_Y-P_1)\Bigl(\sum_{n=1}^{m+1} \alpha_nb_n\Bigr)\biggr\rVert\\ &\le \biggl\lVert \sum_{n=1}^{m+1} \alpha_n b_n\biggr\rVert\qquad (m\in\N,\,\alpha_1,\ldots,\alpha_{m+1}\in\K) \end{align*}
by~\eqref{L:perturb:eq1} and bimonotonicity. Hence~$Y$ admits shift operators $R,L\in\mathscr{B}(Y)$, both having norm~$1$, and  
\begin{equation}\label{L:perturb:eq2} L^nR^n = I_Y\qquad\text{and}\qquad R^nL^n = I_Y-P_n\qquad (n\in\N). \end{equation}

Let $C>1$, and suppose that $S =I_Y-T$ is compact. Then  $P_nS\to S$ as $n\to\infty$, so we can find $n\in\N$ such that $\lVert(I_Y - P_n)S\rVert\le 1-\frac1{C}$. Combining this with~\eqref{L:perturb:eq2}, we obtain
\begin{align*} 
\lVert I_Y - L^n T R^n\rVert &= \lVert L^n(I_Y - T)R^n\rVert = \lVert L^n S R^n\rVert\\ &\le \lVert L^n S\rVert = \lVert R^nL^n S\rVert = \lVert (I_Y-P_n)S \rVert\le 1 - \frac{1}{C}.
\end{align*}
This implies that the operator $L^nTR^n$ is invertible by the Neumann series, and its inverse has norm at most~$C$. 
Consequently we can define operators $U = (L^n TR^n)^{-1}L^n\in\mathscr{B}(Y)$ and $V= R^n\in\mathscr{B}(Y)$ with $\lVert U\rVert\le C$ and $\lVert V\rVert =1$, and $UTV = I_Y$ by definition.
\end{proof}

\begin{remark}\label{R:perturb}
We shall also require the counterpart of \Cref{L:perturb} for $Y = \bigl(\bigoplus_{n\in\N}\ell_2^n\bigr)_E$, where $E=c_0$ or $E=\ell_1$. Clearly~$Y$ has a bimonotone basis, but it does not satisfy~\eqref{L:perturb:eq1}. One could set up a more abstract, unified framework which covers both these spaces and those in \Cref{L:perturb}. However, we find it more instructive to explain directly how to modify the proof of \Cref{L:perturb} to obtain the desired conclusion for $Y = \bigl(\bigoplus_{n\in\N}\ell_2^n\bigr)_E$. 

For $n\in\N$, define $L_n\in\mathscr{B}(\ell_2^{n+1},\ell_2^n)$ and $R_n\in\mathscr{B}(\ell_2^n,\ell_2^{n+1})$ by 
\[ L_n(t_1,\ldots,t_{n+1}) = (t_1,\ldots,t_{n}),\qquad R_n(t_1,\ldots,t_{n}) = (t_1,\ldots,t_n,0)\qquad (t_1,\ldots,t_{n+1}\in\K), \] 
and then define $L,R\in\mathscr{B}(Y)$ by 
\[ Ly = (L_ny_{n+1})_{n\in\N},\qquad Ry = (0,R_1y_1,R_2y_2,\ldots)\qquad (y=(y_n)_{n\in\N}\in Y). \]
As in the proof of \Cref{L:perturb}, these operators satisfy that $R$ is an isometry, $L$ has norm~$1$, $L^nR^n=I_Y$ for each $n\in\N$, and $L^nS\to 0$ as $n\to\infty$ for every $S\in\mathscr{K}(Y)$. Using these facts, we can now repeat the final part of the argument verbatim: Let $C>1$, and suppose that $T\in\mathscr{B}(Y)$ is an operator for which $S=I_Y-T$ is compact. Then, choosing $n\in\N$ such that $\lVert L^nS\rVert\le 1-\frac1{C}$, we have
\[ \lVert I_Y - L^n T R^n\rVert = \lVert L^nSR^n\rVert\le \lVert L^n S\rVert\,\lVert R\rVert^n\le 1 - \frac{1}{C}, \]
so the operator $L^n T R^n$ is invertible, and $\lVert (L^nTR^n)^{-1}\rVert\le C$. Hence the operators $U = (L^n TR^n)^{-1}L^n\in\mathscr{B}(Y)$ and $V= R^n\in\mathscr{B}(Y)$ satisfy $UTV = I_Y$ and $\lVert U\rVert\, \lVert V\rVert\le C$, as desired. 
\end{remark}

\begin{remark}\label{R:AHandFriends} The closed ideals of~$\mathscr{B}(X)$ have been classified for several other Banach spaces~$X$ than those specified in \Cref{mainthm}. All of these Banach spaces originate from the famous Argyros--Haydon space~$X_{\normalfont{\text{AH}}}$, which was created to solve the scalar-plus-compact problem~\cite{ah}. We do not know whether~$\mathscr{B}(X_{\normalfont{\text{AH}}})$ has a unique algebra norm because we do not know whether  every homomorphism from~$\mathscr{B}(X_{\normalfont{\text{AH}}})$ into a Banach algebra is continuous. This question is also open for the variants of~$X_{\normalfont{\text{AH}}}$ whose closed operator ideals have been classified in \cite[Theorem~2.1]{Tarb} and \cite[Theorem~1.4]{KLindiana}, respectively. (We shall consider the latter Banach space in more detail in \Cref{ex:nonuniformincompress} below.) 

 The situation is even more intriguing for the family of Banach spaces~$X_{K}$, where~$K$ is an infinite compact metric space, having the property that the Calkin algebra $\mathscr{B}(X_K)/\mathscr{K}(X_K)$ is isometrically isomorphic to the Banach algebra~$C(K)$ of continuous, scalar-valued functions on~$K$. A Banach space~$X_K$ with this property was originally constructed by Motakis, Puglisi and Zi\-zi\-mo\-pou\-lou~\cite{mpz} in the case where~$K$ is countably infinite,  while Motakis~\cite{Mo} has subsequently found a new approach which handles the general case.
 
 It is undecidable in \textsf{ZFC} whether every homomorphism from~$C(K)$ into a Banach alge\-bra is continuous;  more precisely, Dales~\cite{Da} and Esterle~\cite{Es} independently proved that discontinuous homomorphisms from~$C(K)$  exist under the Continuum Hypothesis, whereas Solovay and Woodin constructed a different model of \textsf{ZFC} in which all such homomorphisms are  continuous; see~\cite{DalesWoodin} for an exposition of this result.   Consequently  it is also undecidable  in \textsf{ZFC} whether every homomorphism from~$\mathscr{B}(X_K)$ into a Banach algebra is continuous. Note that we have a complete classification of the closed ideals of~$\mathscr{B}(X_K)$, as described in \cite[Remark~1.5(vi)]{KLindiana} and \cite[Proposition~8.6]{Mo}. 
\end{remark}

We conclude this section with three examples which use results from the literature to illustrate how quotient algebras of $\mathscr{B}(X)$ may fail to possess certain combinations of the properties we consider. The first two examples concern the \emph{weak Calkin algebra,} that is, the quotient $\mathscr{B}(X)/\mathscr{W}(X)$, where $\mathscr{W}(X)$ denotes the ideal of weakly compact operators on~$X$. 

\begin{example}\label{GST:Thm2.6}
  Let $X$ be a Banach space, identified with its canonical image in the bi\-dual~$X^{**}$. Since an operator $T\in\mathscr{B}(X)$ is weakly compact if and only if $T^{**}[X^{**}]\subseteq X$, the formula
  \[ \varphi(T+\mathscr{W}(X))(x^{**}+X) = T^{**}(x^{**}) + X\qquad (x^{**}\in X^{**},\, T\in\mathscr{B}(X)) \]
  defines an injective and contractive homomorphism $\varphi\colon\mathscr{B}(X)/\mathscr{W}(X)\to\mathscr{B}(X^{**}/X)$, which is non-zero if and only if~$X$ is non-reflexive. Gonz\'{a}lez, Saksman and Tylli studied this homomorphism in~\cite{GST}, showing in particular that when $X = (J_2\oplus J_2\oplus\cdots)_{\ell_2}$ is the $\ell_2$\nobreakdash-sum of countably many copies of James' quasi-reflexive Banach space~$J_2$, the range of~$\varphi$ is not closed \cite[Theorem~2.6]{GST}. Hence the weak Calkin algebra $\mathscr{B}(X)/\mathscr{W}(X)$ fails to be incompressible in this case. Clearly \mbox{$X\cong X\oplus X$}, so \Cref{maximalnorm} and \Cref{T:BEJ} imply that the quotient norm on~$\mathscr{B}(X)/\mathscr{W}(X)$ is maximal.  
\end{example}

Our second example involves the following piece of notation. For a Banach space~$E$, we write $E^{\sim}$ for the Banach algebra which is the unitization of~$E$ equipped with the trivial product; that is, $E^{\sim} = E\times\K$ as a vector space, with the norm and product defined by
\[ \lVert (x,s)\rVert = \lVert x \rVert  +
\lvert s\rvert\qquad \text{and}\qquad (x,s)(y,t) = (sy + tx, st)  \qquad (x,y\in E,\, s,t\in\mathbb{K}), \]
respectively. Given a linear map $\theta\colon E\to F$ into some Banach space~$F$, we define its unitization $\widetilde{\theta}\colon E^{\sim}\to F^{\sim}$ by $\widetilde{\theta}(x,s) = (\theta(x),s)$ for $x\in E$ and $s\in\K$. This is a unital homomorphism which is continuous if and only if~$\theta$ is continuous. 

\begin{example}\label{Ex:Readspace} Let~$X_{\text{R}}$ denote the Banach space which Read constructed in~\cite{read} with the following remarkable property: the weak Calkin algebra of~$X_{\text{R}}$ is isomorphic as a Banach algebra to the unitization~$\ell_2^{\sim}$ of the Hilbert space~$\ell_2$ equipped with the trivial product. (In fact, $X_{\normalfont{\text{R}}}$ has a slightly stronger property, namely: there exists a continuous, surjective homo\-mor\-phism $\psi\colon\mathscr{B}(X_{\normalfont{\text{R}}})\to\ell_2^{\sim}$
with $\ker\psi = \mathscr{W}(X_{\normalfont{\text{R}}})$ such that the extension
\begin{equation*}%\label{WEBEsplitexactEq1} 
 \spreaddiagramcolumns{2ex}\xymatrix{\{0\}\ar[r] &
   \mathscr{W}(X_{\normalfont{\text{R}}})\ar[r] &
   \mathscr{B}(X_{\normalfont{\text{R}}})\ar^-{\displaystyle{\psi}}[r]
   & \ell_2^{\sim}\ar[r] & \{0\}}
\end{equation*} 
splits in the category of Banach algebras, as shown in \cite[Theorem~1.2]{LS2}.) 

We claim that the norm on~$\ell_2^{\sim}$ is neither maximal nor incompressible, and hence the same is true for the weak Calkin algebra $\mathscr{B}(X_{\normalfont{\text{R}}})/\mathscr{W}(X_{\normalfont{\text{R}}})$. 

To see that the norm on~$\ell_2^{\sim}$ fails to be maximal, we must show that there exists a discontinuous homomorphism from~$\ell_2^{\sim}$ into a Banach algebra. This is  well known and not difficult to verify. Indeed, since~$\ell_2$ is infinite-dimensional, it admits a discontinuous linear functional $\theta\colon\ell_2\to\K$, so its unitization $\widetilde{\theta}\colon\ell_2^{\sim}\to\K^{\sim}$ is a discontinuous homomorphism. 

To verify that~$\ell_2^{\sim}$ fails to be incompressible, take $p\in(2,\infty)$ and consider the unitization $\widetilde{\varphi}\colon\ell_2^{\sim}\to\ell_p^{\sim}$ of the formal inclusion map~$\varphi$ from~$\ell_2$ into~$\ell_p$. This is a continuous, injective homomorphism with dense range, but it is not surjective, so it cannot be bounded below. (In fact, it is strictly singular, as observed in \cite[page~76]{LT1}, for instance.) 
\end{example}

Our third and final example displays a  Calkin algebra which is incompressible without being uniformly incompressible.

\begin{example}\label{ex:nonuniformincompress} This example is based on the Banach space~$Z$ studied in~\cite{KLindiana}. It is given by
\begin{equation}\label{D:Z}
Z = X_{\normalfont{\text{AH}}}\oplus_\infty Y, 
\end{equation}
where~$X_{\normalfont{\text{AH}}}$ denotes the Banach space of Argyros and Haydon already mentioned in \Cref{R:AHandFriends}, $Y$~is a certain closed, infinite-dimensional subspace of~$X_{\normalfont{\text{AH}}}$ constructed in \cite[Theorem~1.2]{KLindiana}, and the subscript~$\infty$ indicates that we equip~$Z$ with the norm \[ \lVert(x,y)\rVert = \max\{\lVert x\rVert,\lVert y\rVert\}\qquad (x\in X_{\normalfont{\text{AH}}},\,y\in Y). \]
The key property of~$Z$ that we require is that, 
according to \cite[Equation~(1.2)]{KLindiana},  every operator $T\in\mathscr{B}(Z)$ has the form
\begin{equation}\label{ex:nonuniformincompress:eq1} T = \begin{pmatrix} \alpha_{1,1}I_{X_{\normalfont{\text{AH}}}} & \alpha_{1,2}J\\ 0 & \alpha_{2,2}I_Y \end{pmatrix} + S, \end{equation}
  where  $\alpha_{1,1}, \alpha_{1,2}, \alpha_{2,2}\in\K$, 
  $J\colon Y\to X_{\normalfont{\text{AH}}}$ denotes the inclusion map and $S\in\mathscr{K}(Z)$. (We note in passing that~$Z$ has a basis, so the ideal of finite-rank operators is dense in~$\mathscr{K}(Z)$.)

  Johnson and Phillips~\cite{J-IHABB} observed that the subalgebra
  \[ \left\{ \begin{pmatrix} \alpha & \beta\\ 0 & \alpha \end{pmatrix} : \alpha,\beta\in\K\right\} \]
  of the algebra~$M_2(\K)$  of scalar-valued  $2\times 2$ matrices, equipped with the spectral norm,  is incompressible, but not uniformly incompressible. Building on this example,   we shall show that the same conclusion holds true for the Calkin algebra~$\mathscr{B}(Z)/\mathscr{K}(Z)$.

  First, \eqref{ex:nonuniformincompress:eq1} shows that~$\mathscr{B}(Z)/\mathscr{K}(Z)$ is finite-dimensional and therefore incompressible.
  
Second, we observe that the scalars~$\alpha_{1,1}$, $\alpha_{1,2}$ and~$\alpha_{2,2}$ in~\eqref{ex:nonuniformincompress:eq1} are uniquely determined by~$T$ because the operators~$I_{X_{\normalfont{\text{AH}}}}$, $J$ and~$I_Y$ are non-compact, so for each $\delta\in (0,1)$, we can define a map $\phi_\delta\colon\mathscr{B}(Z)\to M_2(\K)$ by 
  \[ \phi_\delta(T) = \begin{pmatrix} \alpha_{1,1} & \delta\alpha_{1,2}\\ 0 & \alpha_{2,2} \end{pmatrix}. \]
Straightforward calculations show that~$\phi_\delta$ is a unital homomorphism. Furthermore, $\phi_\delta$ is continuous with norm~$1$ provided that we equip~$M_2(\K)$ with the norm induced by identifying it with the Banach algebra~$\mathscr{B}(\ell_\infty^2)$, that is,
  \[ \left\lVert  \begin{pmatrix} \alpha_{1,1} & \alpha_{1,2}\\ \alpha_{2,1} & \alpha_{2,2} \end{pmatrix}\right\rVert = \max\bigl\{\lvert\alpha_{1,1}\rvert +\lvert\alpha_{1,2}\rvert, \lvert\alpha_{2,1}\rvert +\lvert\alpha_{2,2}\rvert\bigr\}. \]   
 Since $\ker\phi_\delta =\mathscr{K}(Z)$, the Fundamental Isomorphism Theorem implies that~$\phi_\delta$ induces a continuous, injective, unital  homomorphism $\widetilde{\phi}_\delta\colon\mathscr{B}(Z)/\mathscr{K}(Z)\to M_2(\K)$, also of norm~$1$, by the formula  $\widetilde{\phi}_\delta(T+\mathscr{K}(Z)) = \phi_\delta(T)$. 

 Consequently, if~$\mathscr{B}(Z)/\mathscr{K}(Z)$ were uniformly incompressible, there would be a constant~$\eta>0$ such that~$\widetilde{\phi}_\delta$  is bounded below by~$\eta$ for every $\delta\in(0,1)$. However, we have
\[ \left\lVert  \begin{pmatrix} 0 & J\\ 0 & 0\end{pmatrix}+\mathscr{K}(Z)\right\rVert =\lVert J+\mathscr{K}(Y,X_{\normalfont{\text{AH}}})\rVert = 1 \] because~$Y$ is infinite-dimensional, and therefore
\[ \eta\le \left\lVert  \widetilde{\phi}_\delta\left(\begin{pmatrix} 0 & J\\ 0 & 0\end{pmatrix}+\mathscr{K}(Z)\right)\right\rVert =   \left\lVert  \begin{pmatrix} 0 & \delta\\ 0 & 0\end{pmatrix} \right\rVert =  \delta\qquad (\delta\in(0,1)),  \]
contradicting that the same constant~$\eta>0$ should work for every~$\delta\in(0,1)$. This contradiction proves that  the Calkin algebra~$\mathscr{B}(Z)/\mathscr{K}(Z)$ fails to be uniformly incompressible. 

\Cref{P:JPSlemma0.2} implies that~$\mathscr{B}(Z)/\mathscr{K}(Z)$ cannot have the $C$-IFP for any $C\ge 1$. This is also clear from the fact that the inclusion operator~$J$ is inessential, as observed in~\cite{KLindiana}, because every inessential, idempotent operator has finite rank. 

A similar conclusion can be drawn for the Banach space $X=\ell_p\oplus\ell_q$ for $1\le p<q<\infty$. As already mentioned on page~\pageref{Ware}, Ware \cite[Corollary~5.2.3]{GW} showed that the Calkin algebra of~$X$ has a unique algebra norm and is therefore incompressible. However, $\mathscr{B}(X)$ contains many non-compact operators which no idempotent operator of infinite rank factors through because \cite[Corollaries~9 and~12]{FSZ} shows that there are $2^{\mathfrak{c}}$ many closed ideals between the ideals of compact and strictly singular operators on~$X$, and every strictly singular, idempotent operator has finite rank. Hence the Calkin algebra of~$X$ cannot have the $C$-IFP for any $C\ge 1$.
\end{example}

\section{A quantitative factorization result for the identity on $c_0$}\label{Section3}
\noindent  
The aim of this section is to establish the following characterization of when the identity operator on~$c_0$ factors through an operator defined on a Banach space of the form~$C_0(K)$ for a locally compact Hausdorff space~$K$.

\begin{theorem}\label{P:AequivB}
    Let $T\in\mathscr{B}(C_0(K),Y)$, where~$K$ is a locally compact Hausdorff space and~$Y$ is a Banach space for which the unit ball of~$Y^*$ is weak* sequentially compact, and let $C>0$. Then $C_0(K)$~con\-tains a sequence $(f_n)_{n\in\N}$ such that 
    \begin{equation}\label{P:AequivB:eq1} 
    \sup_{k\in K} \sum_{n=1}^\infty\lvert f_n(k)\rvert\le 1\qquad\text{and}\qquad \inf_{n\in\N}\lVert Tf_n\rVert >\frac1C
\end{equation} 
if and only if there are operators $U\in\mcb(Y,c_0)$ and $V\in\mathscr{B}(c_0,C_0(K))$ such that  
\begin{equation}\label{P:AequivB:eq2} 
UTV=I_{c_0}\qquad \text{and}\qquad \lVert U\rVert\,\lVert V\rVert<C.
\end{equation}
\end{theorem}

Our main application of \Cref{P:AequivB} will be in \Cref{c0KA}, where it will provide the foundation for the proof of \Cref{mainthm}\ref{mainthm2}. However, a simple special case of it will also help us establish \Cref{mainthm}\ref{mainthm1.5}, which is why we state and prove it at this point.

To aid the presentation of the proof of \Cref{P:AequivB}, we have split it into a number of simpler statements, beginning with one which explains the significance of the left-hand part of~\eqref{P:AequivB:eq1}. For an index set~$\Gamma$, we write $(e_\gamma)_{\gamma\in\Gamma}$ for the unit vector basis of~$c_0(\Gamma)$.

\begin{lemma}\label{L:operatorc0toCK} Let $(f_n)_{n\in\N}$ be a sequence in~$C_0(K)$ for some locally compact Hausdorff space~$K$.  There is an operator $V\in\mathscr{B}(c_0,C_0(K))$ such that $Ve_n = f_n$ for every $n\in\N$ if and only if
\begin{equation}\label{L:operatorc0toCK:eq1} 
   \sup_{k\in K} \sum_{n=1}^\infty\lvert f_n(k)\rvert <\infty. \end{equation}
If one, and hence both, of these conditions are satisfied, the norm of the operator~$V$ is equal to the supremum~\eqref{L:operatorc0toCK:eq1}. 
\end{lemma}

\begin{proof} Let~$C\in[0,\infty]$ denote the supremum in~\eqref{L:operatorc0toCK:eq1}, and observe that 
  \[ C =\sup\biggl\{\sum_{n=1}^m\lvert f_n(k)\rvert : k\in K,\,m\in\N\biggr\}. \]
  
  $\Rightarrow$. Suppose that  $V\in\mathscr{B}(c_0,C_0(K))$ is an operator with $Ve_n = f_n$ for every $n\in\N$, and take $m\in\N$ and $k\in K$. For each $n\in\{1,\ldots,m\}$, choose $\sigma_n\in\K$ such that $\lvert\sigma_n\rvert =1$ and $\sigma_nf_n(k)\ge 0$. Then $x=\sum_{n=1}^m\sigma_n e_n\in c_0$ has norm~$1$, so
  \[ \lVert V\rVert\ge \lVert Vx\rVert_\infty = \biggl\lVert \sum_{n=1}^m\sigma_n f_n\biggr\rVert_\infty\ge \biggl\lvert \sum_{n=1}^m\sigma_n f_n(k)\biggr\rvert  = \sum_{n=1}^m\sigma_n f_n(k) = \sum_{n=1}^m\lvert f_n(k)\rvert. \]
  This proves that $C\le\lVert V\rVert<\infty$.

  $\Leftarrow$. Suppose that the supremum~$C$ is finite.  Since~$c_{00}$ is dense in~$c_0$, it suffices to show that the linear map  $V\colon c_{00}\to C_0(K)$  given by $Ve_n = f_n$ for every $n\in\N$  is bounded. To verify this, we observe that, for $m\in\N$ and $\alpha_1,\ldots,\alpha_m\in\K$ with $\max_{1\le n\le m}\lvert\alpha_n\rvert\le 1$,
  \[ \biggl\lVert V\Bigl(\sum_{n=1}^m\alpha_n e_n\Bigr)\biggr\rVert_\infty =  \biggl\lVert \sum_{n=1}^m\alpha_n f_n\biggr\rVert_\infty = \sup_{k\in K}\biggl\lvert \sum_{n=1}^m\alpha_n f_n(k)\biggr\rvert\le \sup_{k\in K}\sum_{n=1}^m\lvert\alpha_n\rvert\,\lvert f_n(k)\rvert\le C. \]
Hence $V$ is bounded with  $\lVert V\rVert\le C$.   The final clause follows by combining the estimates obtained in the two parts of the proof. 
\end{proof}

We now come to the two main ingredients in the proof of \Cref{P:AequivB}. Both are generalizations of known results and follow from careful examination of the published proofs, as we shall explain in more detail below. The first is a  quantitative version of a celebrated  theorem of  Rosen\-thal, originally stated as the first remark following  \cite[Theorem~3.4]{ROS}. 

\begin{theorem}[Rosenthal]\label{Rosenthal}   
Let $T\in \mcb(c_0(\Gamma),X)$, where~$\Gamma$ is an infinite set and~$X$ is a Banach space, and  suppose that $\delta := \inf_{\gamma \in \Gamma}\|Te_\gamma\| > 0$. Then, for every $\eta\in(0,\delta)$, there is a subset $\Gamma'$ of $\Gamma$ of the same cardinality as $\Gamma$ for which the restriction of $T$ to the sub\-space $\overline{\operatorname{span}}\{e_\gamma: \gamma \in \Gamma'\}$ is bounded below by~$\eta$.
\end{theorem}

\begin{proof}
Rosenthal stated the result without specifying for which values of~$\eta>0$ we can find a subset~$\Gamma'$  of the same cardinality as $\Gamma$ such that the restriction of~$T$ to  $\overline{\operatorname{span}}\{e_\gamma: \gamma \in \Gamma'\}$ is bounded below by~$\eta$.  However, inspection of his proof shows that this is possible for every $\eta<\delta$. 
\end{proof}

The second result generalizes a theorem of Dowling, Randrianantoanina and Turett \cite[Theorem~6]{DoRaTu}. To state it concisely, we require the following notion.

\begin{definition}\label{Cisom}
Let $C\ge 1$. Two Banach spaces~$X$ and~$Y$ are $C$\emph{-isomorphic} if there exists an isomorphism  $T\in\mathscr{B}(X,Y)$ with $\|T\|\,\|T^{-1}\| \leq C$.
\end{definition}

\begin{theorem}[Dowling, Randrianantoanina and Turett]\label{DRT} 
Let $Z$ be a closed subspace of a Banach space~$Y$ for which the unit ball of~$Y^*$ is weak* sequentially compact, and suppose that $Z$ contains a closed sub\-space which is isomorphic to $c_0$. Then, for every $C > 1$, there is a projection $P \in \mcb(Y)$ with $\|P\| \leq C$ such that~$P[Y]$ is contained in~$Z$ and $C$\nobreakdash-iso\-mor\-phic to~$c_0$.
\end{theorem}

\begin{proof}
For  $Z=Y$, this is precisely  the result which Dowling, Randrianantoanina and Turett proved in  \cite[Theorem~6]{DoRaTu}. In order to adapt their proof to the setting where $Z\subsetneq Y$, it suffices to observe that they define the projection~$P$ using Hahn--Banach extensions of the coordinate functionals corresponding to a $(1-\delta)^{-1}$-isomorphic copy of~$c_0$ inside~$Z$, for a suitably defined $\delta\in(0,1)$. By extending these coordinate functionals to all of~$Y$, rather than only to~$Z$, we can follow the rest of their proof verbatim to obtain the stated result. 
\end{proof}

\begin{remark}\label{remark_DRT}
In the case of real scalars, Galego and Plichko \cite[Theorem~4.3]{GP} have proved a result similar to  \Cref{DRT} under the weaker hypothesis that~$Y$ does not contain a copy of~$\ell_1$. However,  it is not clear to us whether this result carries over to complex scalars. Therefore we have opted for the above version, which will suffice for our purposes. 

We note in passing that Galego and Plichko do not state explicitly that the $C$-com\-ple\-mented copy of~$c_0$ they construct inside~$Z$ is $C$-isomorphic to~$c_0$. However, a close inspection of their proof reveals that it is.
\end{remark}

\begin{lemma}\label{L:BA195}
Let $X$ and $Y$ be Banach spaces, where~$X$ contains a subspace which is iso\-morphic to~$c_0$ and the unit ball of~$Y^*$ is weak* sequentially compact, and let $R\in\mathscr{B}(X,Y)$ be an operator which is bounded below by $\eta>0$. Then, for every $C>1$, $R[X]$~contains a closed subspace~$Z$ which is $C$-complemented in~$Y$ and $C$-isomorphic to~$c_0$, and there are operators $U\in\mathscr{B}(Y,c_0)$ and $V\in\mathscr{B}(c_0,X)$ such that $URV = I_{c_0}$ and $\lVert U\rVert\,\lVert V\rVert\le C^2/\eta$. \end{lemma}  

\begin{proof}
  Since~$R$ is bounded below and~$X$ contains a subspace which is isomorphic to~$c_0$, we can apply \Cref{DRT} to the closed subspace~$Z=R[X]$ of~$Y$ to find a projection~$P$ of~$Y$ onto a closed subspace~$Y_0$ of~$R[X]$ such that $\lVert P\rVert\leq C$ and there is an isomorphism $S\in\mathscr{B}(c_0,Y_0)$ with $\lVert S\rVert\,\lVert S^{-1}\rVert\le C$. Set  $X_0 = R^{-1}[Y_0]$, and let $R_0\in\mathscr{B}(X_0,Y_0)$ be the restriction of~$R$, which is an isomorphism with $\lVert R_0^{-1}\rVert\le 1/{\eta}$ because~$R$ is bounded below by~$\eta$. Then we have a commutative diagram
\begin{equation*}%\label{L:factorizationdiagram:eq1}
    \begin{gathered}
    \spreaddiagramrows{2ex}\spreaddiagramcolumns{5ex}% 
    \xymatrix{%
    c_0\ar^-{\displaystyle{I_{c_0}}}[rr] \ar_-{\displaystyle{S}}[d] && c_0\\ 
    Y_0 \ar^-{\displaystyle{R_0^{-1}}}[r] & X_0\ar^-{\displaystyle{R_0}}[r] \ar_-{\displaystyle{J}}[d] & Y_0\ar_-{\displaystyle{S^{-1}}}[u]\\ 
    & X\ar^-{\displaystyle{R}}[r] & Y\ar_-{\displaystyle{P}}[u]\smashw{,}}
    \end{gathered}
    \end{equation*}
    where $J\colon X_0\to X$ denotes the inclusion map. Hence the operators $U = S^{-1}P\in\mathscr{B}(Y,c_0)$ and $V = J R_0^{-1}S\in\mathscr{B}(c_0,X)$ satisfy $URV = I_{c_0}$, and 
    \begin{equation*}%\label{R:IDc0Factorization:eq2} 
      \lVert U\rVert\,\lVert V\rVert\le \lVert S^{-1}\rVert\,\lVert P\rVert\,\lVert J\rVert\,\lVert R_0^{-1}\rVert\,\lVert S\rVert\le\frac{C^2}{\eta}. \qedhere
    \end{equation*}
\end{proof}  

\begin{proof}[Proof of Theorem~{\normalfont{\ref{P:AequivB}}}] $\Rightarrow$. Suppose that~$(f_n)_{n\in\N}$ is a sequence in~$C_0(K)$ which satisfies~\eqref{P:AequivB:eq1}.  \Cref{L:operatorc0toCK} shows that we can define an operator $V_1\in\mathscr{B}(c_0,C_0(K))$ with $\lVert V_1\rVert\le 1$ by $V_1e_n = f_n$ for every $n\in\N$.  Choose $\eta\in (C^{-1},\inf_{n\in\N}\lVert Tf_n\rVert)$. Then \Cref{Rosenthal} implies that~$\N$ contains an infinite subset~$N$ such that the restriction~$R$ of the operator~$TV_1$ to the sub\-space $X=\overline{\operatorname{span}}\{ e_n : n\in N\}$ is bounded below by~$\eta$. 
  
Take $\xi\in(1,\sqrt{C\eta})$.  Since~$X$ is (isometrically) isomorphic to~$c_0$, we can apply \Cref{L:BA195} to obtain opera\-tors $U\in\mathscr{B}(Y,c_0)$ and $V_2\in\mathscr{B}(c_0,X)$ such that \[ URV_2 = I_{c_0}\qquad\text{and}\qquad \lVert U\rVert\,\lVert V_2\rVert\le \frac{\xi^2}{\eta}<C. \] It follows that~\eqref{P:AequivB:eq2} is satisfied provided that we define $V=V_1JV_2\in\mathscr{B}(c_0,C_0(K))$, where $J\colon X\to c_0$ denotes the inclusion map.

Conversely, suppose that $U\in\mcb(Y,c_0)$ and $V\in\mathscr{B}(c_0,C_0(K))$ are operators satisfying~\eqref{P:AequivB:eq2}, and set $f_n = \frac1{\lVert V\rVert} Ve_n\in C_0(K)$ for every $n\in\N$. 
\Cref{L:operatorc0toCK} shows that $\sup_{k\in K} \sum_{n=1}^\infty\lvert f_n(k)\rvert =1$ because the operator $V/\lVert V\rVert$ has norm~$1$. Furthermore, since $UTf_n=\frac1{\lVert V\rVert} e_n$  for every $n\in\N$, we have 
\[ \inf_{n\in\N}\lVert Tf_n\rVert\ge \frac1{\lVert U\rVert\,\lVert V\rVert} >\frac1C, \] 
so the sequence $(f_n)$ satisfies~\eqref{P:AequivB:eq1}.
\end{proof}

\section{The proof of Theorem~\ref{mainthm}\ref{mainthm1}}\label{c0sum}
\noindent 
For Banach spaces~$E$, $X$ and~$Y$, let~$\overline{\mcg}_E(X,Y)$ denote the closure of the set of operators $T\in\mathscr{B}(X,Y)$ which factor through~$E$ in the sense that $T=UV$ for some operators $U\in\mcb(E,Y)$ and $V\in\mcb(X,E)$. This set is closed under addition provided that~$E$ contains a complemented subspace isomorphic to~$E\oplus E$, and therefore~$\overline{\mcg}_E$ defines a closed operator ideal in this case. 

In the remainder of this section, set $E=c_0$ or $E=\ell_1$ (so that~$\overline{\mcg}_E$ is indeed a closed operator ideal), and let~$X$ denote either of the Banach spaces $X=\bigl(\bigoplus_{n \in \mathbb{N}} \ell_2^n\bigr)_E$, the latter being the dual space of the former. The lattice of closed ideals of~$\mathscr{B}(X)$, originally given as \cite[Corollary~5.6]{LLR} for $E=c_0$ and as \cite[Theorem~2.12]{LSZ} for $E=\ell_1$, is
\begin{equation}\label{c0sum:eq}
  \{0\}\subsetneq \mck(X)\subsetneq \overline{\mcg}_E(X) \subsetneq \mcb(X). \end{equation}
A result of Ware \cite[Theorem~5.3.1]{GW}, already mentioned in the second bullet point of clause~\ref{Ware} on page~\pageref{Ware},  tells us that  the Calkin algebra $\mcb(X)/\mck(X)$ has a unique algebra norm for both of these Banach spaces~$X$. (Alternatively, applying \Cref{calkindirectsumincompressible} below with $\Gamma = \emptyset$, we obtain the stronger conclusion that  $\mcb(X)/\mck(X)$ has the $C$-IFP for every $C>1$.) Hence, in view of~\eqref{Proof:todo}, \Cref{mainthm}\ref{mainthm1} will be established once we prove the following  result. 

\begin{theorem}\label{T:LLR_LSZ}
 Let $X=\bigl(\bigoplus_{n \in \mathbb{N}} \ell_2^n\bigr)_E$, where   $E=c_0$ or $E=\ell_1$. 
 For every  constant $C>1$ and every operator $T\in\mathscr{B}(X)\setminus\overline{\mathscr{G}}_E(X)$ with $\lVert T+ \overline{\mathscr{G}}_E(X)\rVert = 1$, there are operators $U,V\in\mathscr{B}(X)$ such that $UTV = I_X$ and $\lVert U\rVert\,\lVert V\rVert\le C$. 
 
 Consequently the quotient alge\-bra $\mathscr{B}(X)\big/\,\overline{\mathscr{G}}_E(X)$  has the $C$-IFP and is uniform\-ly in\-com\-pressible.
\end{theorem}

\begin{notation}\label{N:dirsums} Direct sums and operators between them will play a key role in this section and the next, so to avoid any possible confusion, let us explain the notation we use, all of which is fairly standard. 

  Throughout this paragraph, $\Gamma$ denotes a non-empty set, $(X_\gamma)_{\gamma\in\Gamma}$
  is a family of Banach spaces indexed by~$\Gamma$, and  $E=c_0$ or $E=\ell_1$ as above. We write $\bigl(\bigoplus_{\gamma\in\Gamma} X_\gamma\bigr)_E$ for the \emph{$E$-direct sum} of the family~$(X_\gamma)_{\gamma\in\Gamma}$; that is, $\bigl(\bigoplus_{\gamma\in\Gamma} X_\gamma\bigr)_E$ is the Banach space consisting of those functions $x\colon\Gamma\to\bigcup_{\gamma\in\Gamma}X_\gamma$ such that $x(\gamma)\in X_\gamma$ for each $\gamma\in\Gamma$ and
 \begin{equation}\label{N:dirsums:Eq1}  \begin{cases} \text{the set}\ \{\gamma\in\Gamma : \lvert x(\gamma)\rvert\ge\epsilon\}\ \text{is finite for each}\ \epsilon>0\ &\text{for}\quad E= c_0,\\
     \sum_{\gamma\in\Gamma}\lVert x(\gamma)\rVert <\infty\quad &\text{for}\quad E=\ell_1. \end{cases}
\end{equation}   
The vector-space operations are defined pointwise,  and the norm is given by
\[  \lVert x\rVert = \begin{cases} \sup_{\gamma\in\Gamma}\lVert x(\gamma)\rVert\quad &\text{for}\quad E=c_0,\\ \sum_{\gamma\in\Gamma} \lVert x(\gamma)\rVert\quad &\text{for}\quad E=\ell_1.  \end{cases} \]
(Of course,  \eqref{N:dirsums:Eq1} is trivially satisfied  when~$\Gamma$ is finite; in this case  we simply equip the full Cartesian product of~$(X_\gamma)_{\gamma\in\Gamma}$ with the norm specified above.) We often write elements of $\bigl(\bigoplus_{\gamma\in\Gamma} X_\gamma\bigr)_E$ in the form 
$(x_\gamma)_{\gamma\in\Gamma}$, where $x_\gamma = x(\gamma)$. 
Throughout this section and the next, for $\beta\in\Gamma$, we reserve the symbols $J_\beta\colon X_\beta\to \bigl(\bigoplus_{\gamma\in\Gamma} X_\gamma\bigr)_E$ and \mbox{$Q_\beta\colon \bigl(\bigoplus_{\gamma\in\Gamma} X_\gamma\bigr)_E\to X_\beta$} for the canonical $\beta^{\text{th}}$ coordinate embedding and projection, respectively.

Given an operator $T\in \mathscr{B}\bigl(\bigl(\bigoplus_{\gamma\in\Gamma} X_\gamma\bigr)_E,\bigl(\bigoplus_{\gamma\in\Gamma} Y_\gamma\bigr)_E\bigr)$, for some family $(Y_\gamma)_{\gamma\in\Gamma}$ of Banach spaces,  we associate with it the $\Gamma\times\Gamma$ matrix $(T_{\beta,\gamma})_{\beta,\gamma\in\Gamma}$  whose $(\beta,\gamma)^{\text{th}}$ coefficient is the operator defined by  $T_{\beta,\gamma} = Q_\beta T J_\gamma\in\mathscr{B}(X_\gamma,Y_\beta)$.
Suppose that~$\Gamma$ is infinite. Following \cite{LLR}, we say that the operator~$T$ has \emph{finite rows} if the set $\{ \gamma\in\Gamma : T_{\beta,\gamma}\ne 0\}$ is finite for each $\beta\in\Gamma$; and analogously $T$ has \emph{finite columns}  if the set $\{ \beta\in\Gamma : T_{\beta,\gamma}\ne 0\}$ is finite for each $\gamma\in\Gamma$. An operator which has both finite rows and finite columns has \emph{locally finite matrix.}

Given a family $\{ T_\gamma\in\mathscr{B}(X_\gamma,Y_\gamma) : \gamma\in\Gamma\}$ of operators such that $\sup_{\gamma\in\Gamma} \lVert T_\gamma\rVert<\infty$, we define the associated \emph{diagonal operator} by
\[ \bigoplus_{\gamma\in\Gamma} T_\gamma\colon\ (x_\gamma)_{\gamma\in\Gamma}\mapsto (T_\gamma x_\gamma)_{\gamma\in\Gamma},\quad \Bigl(\bigoplus_{\gamma\in\Gamma} X_\gamma\Bigr)_E\to \Bigl(\bigoplus_{\gamma\in\Gamma} Y_\gamma\Bigr)_E. \] 
This is clearly a bounded operator with norm $\sup_{\gamma\in\Gamma} \lVert T_\gamma\rVert$. 
\end{notation}

\subsection*{The case {\mathversion{bold}$E=c_0$}.} Throughout this subsection, we consider the Banach space \begin{equation*}%\label{def:Xc_0}
X=\Bigl(\bigoplus_{n \in \mathbb{N}} \ell_2^n\Bigr)_{c_0}. \end{equation*}
The  following index~$ m_\epsilon(T)$, which was originally introduced in \cite[Defini\-tion~5.2(ii)]{LLR}, played a key role  in the classification~\eqref{c0sum:eq} of the closed ideals of~$\mcb(X)$, and it will also be an essential ingredient in our proof of \Cref{T:LLR_LSZ}. 

\begin{definition}\label{defnmepsilon}
For $\epsilon>0$ and an operator $T\in\mathscr{B}\bigl(\bigl(\bigoplus_{n\in N}H_n\bigr)_{c_0},Y)$, where the index set~$N$ is  finite, $H_n$ is a Hilbert space for each $n\in N$ and~$Y$ can be any Banach space, we define
\begin{multline*}
  m_\epsilon(T) = \sup\biggl\{ m \in \mathbb{N}_0 : \Bigl\lVert T\circ \bigoplus_{n\in N} (I_{H_n}-P_{G_n})\Bigr\rVert > \epsilon\ \text{for every subspace}\  G_n\ \text{of}\ H_n\biggr.\\[-4mm] \biggl. \text{with}\ \dim G_n\leq m\ \text{for each}\ n\in N\biggr\}\in\N_0\cup\{\pm\infty\},
\end{multline*}
where $P_{G_n}$ denotes the orthogonal projection of~$H_n$ onto the subspace~$G_n$.
\end{definition}

Loosely speaking, $m_\epsilon(T)$ is the largest number of dimensions that we can remove from each of the Hilbert spaces in the domain of~$T$ and still obtain an operator with norm  greater than~$\epsilon$.

Suppose that $T\in\mcb(X)$ is an operator with finite rows. As described in~\cite[Remark~5.4]{LLR},  the above definition of~$m_\epsilon$ can be applied to define $m_\epsilon(Q_jT)$ for every $j\in\mathbb{N}$ by ignoring the cofinite number of Hilbert spaces in the domain of~$Q_jT$ on which it acts trivially.

The following lemma explains why the quantities $m_\epsilon(Q_jT)$ for $j\in\mathbb{N}$ are relevant in our investigation. It is proved in \cite[Theorem~5.5(iii)]{LLR}. (The bounds on the norms of the ope\-ra\-tors~$U$ and~$V$ are not stated explicitly in~\cite{LLR}, but follow from the proof of the result.) 

\begin{lemma}\label{L:LLR5.5iii} Let $T \in \mcb(X)$ be an operator with locally finite matrix, and suppose that the set $\{m_\epsilon (Q_jT) : j\in \mathbb{N}\}$ is unbounded for some $\epsilon>0$. Then there are operators $U,V\in \mcb(X)$ with $\lVert U\rVert\le 1/\epsilon$ and $\lVert V\rVert\le 1$ such that $UTV=I_X$.
\end{lemma}

In order to apply this result, we require specific values of $\epsilon>0$  for which the set $\{m_\epsilon(Q_jT) : j\in \mathbb{N}\}$ is unbounded. Our next lemma addresses this point.

\begin{lemma}\label{m.epsilon}
Let $T \in \mcb(X)\setminus \overline{\mcg}_{c_0}(X)$ be an operator with locally finite matrix, and suppose that the set $\{ m_\epsilon (Q_j T) : j\in\N\}$ is bounded for some $\epsilon>0$. Then $\epsilon\ge\lVert T+\overline{\mcg}_{c_0}(X)\rVert$.

Consequently the set $\{m_\epsilon (Q_j T) : j\in\mathbb{N}\}$ is unbounded whenever  $0<\epsilon<\lVert T+\overline{\mcg}_{c_0}(X)\rVert$.   
\end{lemma}

\begin{proof}
  Set $m=\sup\{ m_\epsilon (Q_j T) : j\in\N\}\in\N_0$. Our proof uses the machinery from \cite[Construction~4.2]{LLR}, so we begin by introducing the necessary  notation.    For every $j\in\mathbb{N}$,  define $N_j = \{k\in\mathbb{N} : T_{j,k}\neq 0\}$, which is finite, and $B_j = \bigl(\bigoplus_{k\in N_j}\ell_2^k\bigr)_{c_0}$, and 
  let $L_j\colon B_j\to X$ be the natural inclusion map. Furthermore, set $B =\bigl(\bigoplus_{j\in\N} B_j\bigr)_{c_0}$ and  $\widetilde{T} = \bigoplus_{j\in\N} Q_jTL_j\in\mathscr{B}(B,X)$. Then, as shown in \cite[Con\-struc\-tion~4.2]{LLR}, there exists an operator $V\in\mathscr{B}(X,B)$ with $\lVert V\rVert =1$ such that $T= \widetilde{T}V$. 

The convention described after \Cref{defnmepsilon} means that $m_\epsilon(Q_jTL_j) = m_\epsilon(Q_jT)\le m$ for every $j\in\N$. Therefore we can find orthogonal projections $P_{j,k}\in \mcb(\ell_2^k)$  for $k\in N_j$, each having rank at most $m+1$, such that 
\begin{equation}\label{m.epsilon:eq1} \Bigl\lVert Q_jTL_j\circ\bigoplus_{k\in N_j}(I_{\ell_2^k}-P_{j,k})\Bigr\rVert\leq \epsilon. \end{equation}
Set $R_j = \bigoplus_{k\in N_j}P_{j,k}\in \mcb(B_j)$ and $R = \bigoplus_{j\in\N}R_j\in\mcb(B)$. Since each of the projections~$P_{j,k}$ has rank at most $m+1$, the image of~$R$ embeds isometrically into a direct sum of the form $(\ell_2^{m+1}\oplus\ell_2^{m+1}\oplus\cdots)_{c_0}$, which in turn is isomorphic to $(\ell_\infty^{m+1}\oplus\ell_\infty^{m+1}\oplus\cdots)_{c_0}\equiv c_0$. This shows that~$R$ factors through~$c_0$, and therefore the same is true for the composite operator~$\widetilde{T}RV$. Hence we have
\begin{align*}
    \lVert T + \overline{\mcg}_{c_0}(X)\lVert&\leq \lVert T - \widetilde{T}RV \rVert = \lVert\widetilde{T}(I_B-R)V\rVert\le \Bigl\lVert  \bigoplus_{j\in\N} Q_jTL_j(I_{B_j}-R_j)\Bigr\rVert\,\lVert V\rVert\le\epsilon,
\end{align*}
where the final estimate follows because $\lVert V\rVert =1$ and $\lVert Q_jTL_j(I_{B_j}-R_j)\rVert\le\epsilon$ for every $j\in\N$ by~\eqref{m.epsilon:eq1}. The final clause is immediate.
\end{proof}

\begin{proof}[Proof of Theorem~{\normalfont{\ref{T:LLR_LSZ}}} for $E=c_0$] Take a constant $C>1$ and an operator $T\in\mcb(X)$ with $\lVert T + \overline{\mcg}_{c_0}(X)\rVert = 1$. By \cite[Lemma~2.7(iii)]{LLR}, we can find \mbox{$S\in\mck(X)$} such that $T-S$ has locally finite matrix. Since\[ C^{-\frac12}<1=\lVert T + \overline{\mcg}_{c_0}(X)\rVert = \lVert T - S + \overline{\mcg}_{c_0}(X)\rVert, \] \Cref{m.epsilon} shows that the set $\{ m_{C^{-\frac12}}(Q_j(T-S)) : j\in\N\}$ is unbounded. Consequently, applying \Cref{L:LLR5.5iii}, we can find operators $U_1,V_1\in\mcb(X)$ with $\lVert U_1\rVert\leq C^{\frac12}$ and $\lVert V_1\rVert\leq 1$  such that \mbox{$U_1(T-S)V_1 =I_X$}. Since $U_1SV_1$ is compact, \Cref{R:perturb} implies that there are operators $U_2,V_2\in\mathscr{B}(X)$ with $\lVert U_2\rVert\,\lVert V_2\rVert\leq C^{\frac12}$ such that $U_2(U_1TV_1)V_2 = I_X$. Now the con\-clu\-sion follows by taking $U = U_2U_1$ and $V = V_1V_2$.  
\end{proof}

\subsection*{The case {\mathversion{bold}$E=\ell_1$}.} We now turn our attention to the Banach space $X=\bigl(\bigoplus_{n \in \mathbb{N}} \ell_2^n\bigr)_{\ell_1}$. It  identifies naturally with the dual space of the Banach space $\bigl(\bigoplus_{n \in \mathbb{N}} \ell_2^n\bigr)_{c_0}$ that we considered in the previous subsection, and our proof is essentially a dual version of the proof we have just completed. As observed in~\cite[Remark~2.13]{LSZ},  the dualization is not entirely straightforward because we can no longer perturb operators to arrange that they have locally finite matrix; only finite columns can be achieved. Fortunately, that will suffice to carry out the necessary steps, beginning with the following index originally introduced in  \cite[Defi\-ni\-tion~2.4]{LSZ}, which is the dual version  of \Cref{defnmepsilon} above.

\begin{definition}
For $\epsilon>0$ and an operator $T\in\mathscr{B}\bigl(Y,\bigl(\bigoplus_{j\in M}H_j\bigr)_{\ell_1})$, where~$Y$ can be any Banach space, the index set~$M$ is  finite and $H_j$ is a Hilbert space for each $j\in M$, we define
\begin{multline*}
  n_\epsilon(T) = \sup\biggl\{ n\in \mathbb{N}_0 : \biggl\lVert \Bigl(\bigoplus_{j\in M} (I_{H_j}-P_{G_j})\Bigr) T\biggr\rVert > \epsilon\ \text{for every subspace}\ G_j\ \text{of}\ H_j\biggr.\\[-4mm] \biggl.  \text{with}\ \dim G_j\leq n\ \text{for each}\ j\in M\biggr\}\in\N_0\cup\{\pm\infty\},
\end{multline*}
where $P_{G_j}$ denotes the orthogonal projection of~$H_j$ onto the subspace~$G_j$, as before.
\end{definition}

Hence $n_\epsilon(T)$ is the largest number of dimensions that we can remove from each of the Hilbert spaces in the codomain of~$T$ and still obtain an operator with norm greater than~$\epsilon$.

When $T\in \mcb(X)$ has finite columns, we can apply this definition to the operator~$TJ_k$ for every $k\in \mathbb{N}$ by simply ignoring the Hilbert spaces~$\ell_2^j$ in the co\-domain of~$TJ_k$ corresponding to the  co\-finite set of indices~$j\in\N$ for which $Q_jTJ_k=0$. Then, as one would hope, the quantities $n_\epsilon(TJ_k)$ for $k\in\N$ measure how close the identity operator on~$X$ is to factoring through~$T$. More precisely, we have the following counterpart of \Cref{L:LLR5.5iii}, originally proved as part of \cite[Prop\-o\-si\-tion~2.10]{LSZ}. (The norm bounds on~$U$ and~$V$ are not stated explicitly in~\cite{LSZ}, but follow easily from the proof of the implication (ii)$\Rightarrow$(iii).)

\begin{lemma}\label{LSZ2.10iii} Let $T \in \mcb(X)$ be an operator with finite columns, and suppose that the set $\{n_\epsilon (TJ_k) : k\in \mathbb{N}\}$ is unbounded for some $\epsilon>0$. Then there are operators $U,V\in\mcb(X)$ with $\lVert U\rVert\le 1/\epsilon$ and $\lVert V\rVert\le 1$ such that $VTU=I_X$.
\end{lemma}

Our next lemma serves the same purpose as \Cref{m.epsilon}, and its proof is similar.

\begin{lemma}\label{n.epsilon}
  Let $T \in \mcb(X) \setminus \overline{\mcg}_{\ell_1}(X)$ be an operator with finite columns, and suppose that  $\sup\{n_\epsilon(TJ_k) : k\in\mathbb{N}\}$ is finite for some $\epsilon>0$.  Then $\epsilon\ge\lVert T+ \overline{\mcg}_{\ell_1}(X)\rVert$.

  Consequently the set $\{n_\epsilon(TJ_k) : k\in\mathbb{N}\}$ is unbounded whenever $0<\epsilon< \lVert T+ \overline{\mcg}_{\ell_1}(X)\rVert$.
\end{lemma}

\begin{proof} Set $n=\sup\{n_\epsilon(TJ_k) : k\in\mathbb{N}\}\in\N_0$. The set \mbox{$M_k = \{j\in\mathbb{N} : T_{j,k}\neq 0\}$} is finite for each $k\in\mathbb{N}$ because~$T$ has finite columns. Define $Y_k = \bigl(\bigoplus_{j\in M_k}\ell_2^j\bigr)_{\ell_1}$, and
let $L_k\colon Y_k\to X$ and $S_k\colon X\to Y_k$ denote the natural in\-clu\-sion map and projection, respectively. Further\-more, define $Y= \bigl(\bigoplus_{k\in\N} Y_k\bigr)_{\ell_1}$ and $\widetilde{T} = \bigoplus_{k\in\N} S_kTJ_k\in\mathscr{B}(X,Y)$. Then, for each element $y=(y_k)_{k\in\N}\in Y$, the series $\sum_{k\in\N}L_ky_k$ converges absolutely in~$X$, and the operator $W\in\mathscr{B}(Y,X)$ given by $Wy=\sum_{k\in\N}L_ky_k$ satisfies $\lVert W\rVert = 1$ and  $T=W\widetilde{T}$. 

For every $k\in\N$, we have $n\ge n_\epsilon(TJ_k)=n_\epsilon(S_kTJ_k)$, so we can find orthogonal projections $P_{j,k}\in\mcb(\ell_2^j)$ for  $j\in M_k$, each having rank at most $n+1$, such that
\begin{equation}\label{n.epsilon:eq1} \biggl\lVert\Bigl(\bigoplus_{j\in M_k}(I_{\ell_2^j} - P_{j,k})\Bigr)S_kTJ_k\biggr\rVert\le\epsilon. \end{equation}
  Set $R_k=\bigoplus_{j\in M_k}P_{j,k}\in\mathscr{B}(Y_k)$ and $R=\bigoplus_{k\in\N}R_k\in\mathscr{B}(Y)$, and observe that~$R$ factors through~$\ell_1$ because the image of~$R$ embeds isometrically into $(\ell_2^{n+1}\oplus\ell_2^{n+1}\oplus\cdots)_{\ell_1}$, which is isomorphic to $(\ell_1^{n+1}\oplus\ell_1^{n+1}\oplus\cdots)_{\ell_1}\equiv\ell_1$. Consequently we have
  \[ \lVert T+\overline{\mathscr{G}}_{\ell_1}(X)\rVert\le \lVert T - WR\widetilde{T}\rVert = \lVert W(I_Y-R)\widetilde{T}\rVert\le \lVert W\rVert\,\Bigl\lVert\bigoplus_{k\in\N}(I_{Y_k}-R_k)S_kTJ_k\Bigr\rVert\le\epsilon \]
by~\eqref{n.epsilon:eq1}, as required. The final clause is immediate.
\end{proof}

\begin{proof}[Proof of Theorem~{\normalfont{\ref{T:LLR_LSZ}}} for $E=\ell_1$]  Take $C>1$ and $T\in  \mcb(X)$ with $\lVert T+\overline{\mcg}_{\ell_1}(X)\rVert=1$, and apply \cite[Lemma~2.7(i)]{LLR} to find $S\in \mck(X)$ such that $T-S$ has finite columns. Then $C^{-\frac12}<1 = \lVert  T-S+\overline{\mcg}_{\ell_1}(X)\rVert$, and using  Lemmas~\ref{n.epsilon} and~\ref{LSZ2.10iii} as well as \Cref{R:perturb}, we can complete the proof in the same way as we did for $E=c_0$ above.
\end{proof}

\section{The proof of Theorem~\ref{mainthm}\ref{mainthm1.5}}\label{directsum}
\noindent The aim of this section is to show that every quotient of~$\mathscr{B}(X)$ by one of its closed ideals has a unique algebra norm for each of the Banach spaces 
\begin{equation}\label{directsum:eq0}
   X  = \Bigl(\bigoplus_{n\in\mathbb{N}}\ell_2^n\Bigr)_{c_0}\oplus c_0(\Gamma)\qquad\text{and}\qquad X = \Bigl(\bigoplus_{n\in\mathbb{N}}\ell_2^n\Bigr)_{\ell_1}\oplus\ell_1(\Gamma), 
\end{equation}
where the in\-dex set~$\Gamma$ is an uncountable cardinal number. As before, it suffices to show that the quotient of~$\mathscr{B}(X)$ by each of its non-trivial closed ideals is incompressible, and in line with our previous approach, we shall in fact prove the stronger result that each of these quotients has the $C$-IFP for some $C>1$; see \Cref{P:IFPlargeideals} and \Cref{calkindirectsumincompressible} for details. 

The closed ideals of~$\mathscr{B}(X)$ were classified in~\cite{AL}. Unsurprisingly, this result relies heavily on the ideal classifications for each of the two summands of~$X$. The first of these played the main role in \Cref{c0sum} and was stated explicitly in~\eqref{c0sum:eq}.  
Let us now review the second. The following definition is at the heart of it.

\begin{definition}
  Let $\kappa$ be an infinite cardinal number.
  An operator $T\in\mcb(X,Y)$ between  Banach spaces~$X$ and~$Y$ is \emph{$\kappa$-compact} if, for every $\epsilon > 0$, the  closed unit ball~$B_X$ of~$X$ contains a subset $X_\epsilon$ with $\lvert X_\epsilon\rvert<\kappa$ such that
  \[ \inf\bigl\{ \lVert T(x-y)\rVert :  y\in X_\epsilon \bigr\}\leq \epsilon\qquad (x\in B_X). \]
\end{definition}

We write $\mck_\kappa(X,Y)$ for the subset of $\mcb(X,Y)$ consisting of $\kappa$-compact operators. This is a closed operator ideal. The name ``$\kappa$-compact operator'' refers to the fact that it extends the usual notion of compactness for  operators (which corresponds to $\kappa=\aleph_0$) to uncountable cardinal numbers.  We call the quotient algebra $\mathscr{B}(X)/\mathscr{K}_\kappa(X)$ the \emph{$\kappa$-Calkin algebra}.

Generalizing a theorem of Gramsch~\cite{Gr} and Luft~\cite{Luft} in the non-separable Hilbertian case, Daws \cite[Theorem~7.4]{DAWS} showed that  the lattice of closed ideals of $\mcb(X)$ for each of the long sequence spaces $X = c_0(\Gamma)$ and $X=\ell_p(\Gamma)$, where $1\le p<\infty$ and the index set~$\Gamma$ is an uncountable cardinal number, is given by
\begin{equation}\label{Eq:ClosedIdealsofc0Gamma}
  \begin{split}
  \{0\} \subsetneq \mck(X) = \mck_{\aleph_0}(X)\subsetneq \mck_{\aleph_1}(X) \subsetneq &\cdots  \subsetneq \mck_\kappa(X) \subsetneq \mck_{\kappa^+}(X)\subsetneq\cdots\\ &\cdots\subsetneq \mck_{\Gamma}(X) \subsetneq \mck_{\Gamma^+}(X)=\mcb(X),
  \end{split}
\end{equation}  
where $\kappa^+$ denotes the cardinal successor of a cardinal number~$\kappa$. Johnson, Kania and Schecht\-man \cite[Theorem~1.5]{JKS} gave an alternative statement and proof of this result, with additional details provided in \cite[Theorem~3.7]{HK}.
\smallskip

Let~$X_1$ and~$X_2$ be Banach spaces. As in Nota\-tion~\ref{N:dirsums}, we equip their direct sum~$X_1\oplus X_2$ with the $E$-norm, where $E=c_0$ or $E=\ell_1$, and  for $m\in\{1,2\}$, $J_m\colon X_m\to X_1\oplus X_2$ and $Q_m\colon X_1\oplus X_2\to X_m$ denote the $m^{\text{th}}$ coordinate embedding and projection, respectively. 

For a subset $\mci$ of $\mcb(X_1\oplus X_2)$ and $j,k\in \{1,2\}$,  we define the $(j,k)^{\text{th}}$ \textit{quadrant} of~$\mci$ by
\begin{equation*}%\label{Eq:quadrants}
  \mci_{j,k} = \{ Q_jTJ_k : T\in\mci\}\subseteq\mcb(X_k,X_j).
\end{equation*}
Conversely, given subsets $\mci_{j,k}$ of $\mcb(X_k,X_j)$ for  $j,k\in\{1,2\}$, we set
\begin{equation*}%\label{Eq:matrixofsets}
\begin{pmatrix} \mci_{1,1} & \mci_{1,2}\\ \mci_{2,1} & \mci_{2,2}  \end{pmatrix} = \left\{ \begin{pmatrix} T_{1,1} & T_{1,2}\\ T_{2,1} & T_{2,2}
\end{pmatrix} : T_{j,k} \in \mci_{j,k}\ \bigl(j,k\in \{1,2\}\bigr)\right\}\subseteq\mcb(X_1 \oplus X_2).   \end{equation*}
The significance of this notation is due to \cite[Lemma~2.1]{AL}, which states that 
\begin{equation}\label{LemmaDiagonalIdealEq1}  \mci =
  \begin{pmatrix} \mci_{1,1} & \mci_{1,2}\\ \mci_{2,1} & \mci_{2,2}  \end{pmatrix}
\end{equation}
for every ideal~$\mci$ of $\mcb(X_1\oplus X_2)$, and~$\mci_{j,j}$ is an ideal of  $\mcb(X_j)$ for each $j\in\{1,2\}$. Furthermore,   $\mci_{j,k}$ is closed in $\mcb(X_k,X_j)$ for each $j,k\in\{1,2\}$ if and only if~$\mci$ is closed in~$\mcb(X_1\oplus X_2)$.  
 
Now, to state the classification of the closed ideals of~$\mathscr{B}(X)$ for each of the Banach spaces~$X$ given by~\eqref{directsum:eq0} in a unified way and align our notation with that used in the previous paragraph, set $(E,E_\Gamma) = (c_0,c_0(\Gamma))$ or $(E,E_\Gamma) = (\ell_1,\ell_1(\Gamma))$, where~$\Gamma$ is an  uncountable cardinal number, and  observe that
\begin{equation}\label{directsum:eq1}
X = X_1\oplus X_2\qquad\text{for}\qquad X_1 = \Bigl(\bigoplus_{n\in\mathbb{N}}\ell_2^n\Bigr)_E\qquad\text{and}\qquad X_2= E_\Gamma.
\end{equation}
In this notation, the main result of~\cite{AL} states that the  non-trivial closed ideals of~$\mcb(X)$ are 
\begin{equation}\label{directsum:eq3} \mck_{\kappa}(X)\quad (\aleph_0\le\kappa\leq \Gamma)\quad\text{and}\quad
\mcj_\kappa(X)= \begin{pmatrix} \overline{\mcg}_{E}(X_1) & \mcb(X_2,X_1)\\ \mcb(X_1,X_2) & \mck_\kappa(X_2)\end{pmatrix}\quad (\aleph_1\le\kappa\le\Gamma^+). \end{equation}

We split the proof that the quotient of~$\mathscr{B}(X)$ by each of these ideals has the $C$-IFP for some $C>1$ in two cases, beginning with those ideals that are strictly larger than the ideal of compact operators (which we recall is labelled~$\mathscr{K}_{\aleph_0}(X)$ in the list above). This relies on the following elementary general result.

\begin{lemma}\label{IFPlemma2}
Let  $X_1$ and~$X_2$ be Banach spaces, and suppose that~$\mci$ is a closed ideal of~$\mathscr{B}(X_1\oplus X_2)$ which satisfies:
\begin{enumerate}[label={\normalfont{(\roman*)}}]
\item\label{DSjj}  There are constants $C_1,C_2\ge 1$ such that $\mcb(X_j)/\mci_{j,j}$ has the $C_j$-IFP  for $j\in \{1,2\}$.
\item\label{DSjk} $\mci_{j,k} = \mcb(X_k,X_j)$ whenever $j,k \in \{1,2\}$ are distinct. 
\end{enumerate}
Then $\mcb(X_1\oplus X_2)/\mci$ has the $C$-IFP for $C=\max\{C_1,C_2\}$.
\end{lemma}

\begin{proof}
  Take $T = (T_{j,k})_{j,k=1}^2\in\mathscr{B}(X_1\oplus X_2)$ with $\lVert T+\mci\rVert=1$.  Condition~\ref{DSjk} implies that $\lVert T + \mci\rVert  = \max\bigl\{\lVert T_{1,1} + \mci_{1,1} \rVert, \lVert T_{2,2} + \mci_{2,2} \rVert\bigr\}$, so we can choose $j \in \{1,2\}$ such that $\lVert T_{j,j}+\mci_{j,j}\rVert = 1$. Now condition~\ref{DSjj} shows that there are operators $R,S \in \mcb(X_j)$ with $\lVert R+\mci_{j,j}\rVert\, \lVert S+\mci_{j,j}\rVert\leq C_j$ such that $RT_{j,j}S + \mci_{j,j}$ is a non-zero idempotent.
  
  Set $U=J_j R Q_j$ and $V = J_j SQ_j$. Then $U,V\in \mcb(X_1\oplus X_2)$ satisfy \[ \lVert U + \mci \rVert\, \lVert V+\mci\rVert = \lVert R+ \mci_{j,j}\rVert\, \lVert S+\mci_{j,j}\rVert \leq C_j\le \max\{C_1,C_2\}, \] and
$(U+\mci)(T+\mci)(V+\mci) = J_j RT_{j,j}S Q_j + \mci$ is an idempotent, which is non-zero because  $RT_{j,j}S \notin \mci_{j,j}$.
\end{proof}

In order to apply this lemma, we require the following result, which Ware \cite[Propositions~6.3.1--6.3.2]{GW} observed was established in Daws' proofs of \cite[Theorems~6.2 and~7.3]{DAWS}. (We note in passing that we do not know whether the constant~$4$ in~\ref{kappacalkinIFP:1} is optimal.)

\begin{theorem}\label{kappacalkinIFP} Let~$\Gamma$ and~$\kappa$ be uncountable cardinal numbers with $\kappa\le\Gamma$. Then:
\begin{enumerate}[label={\normalfont{(\roman*)}}]
\item\label{kappacalkinIFP:1}   The $\kappa$-Calkin algebras of~$c_0(\Gamma)$ and~$\ell_p(\Gamma)$ for $1<p<\infty$ have the $4$-IFP.
\item\label{kappacalkinIFP:2}  The $\kappa$-Calkin algebra of~$\ell_1(\Gamma)$ has the $C$-IFP for every constant $C>1$.
\end{enumerate}
\end{theorem}

\begin{theorem}\label{P:IFPlargeideals}
    Let~$X$ be given  by~\eqref{directsum:eq0}. Then, for every proper closed ideal~$\mci$ of~$\mcb(X)$ such that~$\mck(X)\subsetneq \mci$, the quotient algebra $\mcb(X)/\mci$ has the $4$-IFP and is therefore uniformly incompressible.
\end{theorem}

\begin{proof}
Inspecting the list~\eqref{directsum:eq3} of proper closed ideals of~$\mathscr{B}(X)$ which are strictly larger than~$\mck(X)$, we see that
  \begin{align}
  \mci_{1,1} &= \begin{cases} \overline{\mathscr{G}}_E(X_1)\ &\text{for}\ \mci =  \mcj_\kappa(X),\ \text{where}\ \aleph_1\le\kappa\le\Gamma^+,\\ \mathscr{B}(X_1)\ &\text{for}\  \mci = \mathscr{K}_\kappa(X),\ \text{where}\ \aleph_1\le\kappa\le\Gamma, \end{cases}\notag\\
   \mci_{2,2} &= \begin{cases} 
    \mathscr{K}_\kappa(X_2)\ &\text{for}\ \mci =  \mck_\kappa(X)\ \text{or}\ \mci =  \mcj_\kappa(X),\ \text{where}\ \aleph_1\le\kappa\le\Gamma,\\
    \mathscr{B}(X_2)\ &\text{for}\  \mci =  \mcj_{\Gamma^+}(X),\end{cases}\notag\\
  \intertext{and}  
  \mci_{j,k}&=\mathscr{B}(X_k,X_j)\qquad (j,k\in\{1,2\},\, j\ne k).\label{P:IFPlargeideals:eq1}    
  \end{align}
  Indeed, this is clear from the definition if $\mci=\mathscr{J}_\kappa(X)$ for some $\aleph_1\le\kappa\le\Gamma^+$. Other\-wise $\mci=\mathscr{K}_\kappa(X)$ for some $\aleph_1\le\kappa\le\Gamma$, so $\mci_{j,k}=\mathscr{K}_\kappa(X_k,X_j)$, and the above identities follow from the fact that every operator with separable range is $\aleph_1$-compact. 
  
  Hence Theorems~\ref{T:LLR_LSZ}  and~\ref{kappacalkinIFP} imply that the quotient algebras $\mathscr{B}(X_j)/\mci_{j,j}$ for $j\in\{1,2\}$ have the 4-IFP whenever they are non-zero, and~\eqref{P:IFPlargeideals:eq1} ensures that \Cref{IFPlemma2} applies. It follows that~$\mathscr{B}(X)/\mci$ has the 4-IFP if $\mathscr{B}(X_1)/\mci_{1,1}$ and $\mathscr{B}(X_2)/\mci_{2,2}$ are both non-zero. Otherwise we have $\mci_{j,j}=\mathscr{B}(X_j)$ for $j=1$ or $j=2$, and $\mathscr{B}(X)/\mci\cong\mathscr{B}(X_k)/\mci_{k,k}$, where $k=2$ if $j=1$ and $k=1$ if $j=2$, so~$\mathscr{B}(X)/\mci$ also has the 4-IFP in this case.  
  \end{proof}

It remains to tackle the Calkin algebra of~$X$. For brevity, we write $\lVert T\rVert_e$ for the essential norm of an operator $T\in\mathscr{B}(X,Y)$; that is, 
\begin{equation}\label{Eq:essentialnorm}
     \lVert T\rVert_e = \lVert T + \mathscr{K}(X,Y)\rVert.
\end{equation}  
Furthermore, for a set~$\Gamma$,  $[\Gamma]^{<\omega}$ denotes the collection of finite subsets of~$\Gamma$. 

\begin{lemma}\label{L:Calkin_c0_ell1_FinDim}
  Let $X = \bigl(\bigoplus_{\gamma\in\Gamma} F_\gamma\bigr)_{E}$ and  $Y = \bigl(\bigoplus_{\xi\in\Xi} G_\xi\bigr)_{E}$, where $E=c_0$ or $E=\ell_1$, the index sets~$\Gamma$ and~$\Xi$ are infinite and~$F_\gamma$ and $G_\xi$ are non-zero, finite-dimensional Banach spaces for every $\gamma\in\Gamma$ and $\xi\in\Xi$. 
Then, for every non-compact operator $T\in\mathscr{B}(X,Y)$ with $\lVert T\rVert_e=1$ and every constant $C>1$, there are operators $U\in\mathscr{B}(Y,E)$ and $V\in\mathscr{B}(E,X)$ such that 
    \begin{equation}\label{L:Calkin_c0_ell1_FinDim:eq0}
      UTV = I_{E}\qquad\text{and}\qquad \lVert U\rVert\,\lVert V\rVert\le C. 
    \end{equation}
\end{lemma}

\begin{proof}
For $A\in[\Gamma]^{<\omega}$, set $R_A = \sum_{\gamma\in A} J_\gamma Q_\gamma\in\mathscr{B}(X)$. 
This is a finite-rank projection, and both~$R_A$ and~$I_X-R_A$ have norm~$1$ (provided $A\ne\emptyset$). (In the second part of the proof, we shall also use this notation for the Banach space~$Y$; this should not cause any confusion.) 

In the case where $E=c_0$, take $\eta\in(C^{-1},1)$. Combining the density of the subspace $\operatorname{span}\{ J_\gamma(F_\gamma) : \gamma\in \Gamma\}$ in~$X$ with the fact that $\lVert T(I_X-R_A)\rVert\ge\lVert T\rVert_e=1$ for each $A\in[\Gamma]^{<\omega}$, we can recursively construct a sequence $(A_n)_{n\in\N}$ of non-empty disjoint sets in~$[\Gamma]^{<\omega}$ such that, for each $n\in\N$, there is a unit vector $x_n\in\operatorname{span}\{ J_\gamma(F_\gamma) : \gamma\in A_n\}$ for which $\lVert Tx_n\rVert > \eta$.
Since the sets~$A_n$ are disjoint and $\lVert x_n\rVert =1$, we can define a linear isometry $V_1\in\mathscr{B}(c_0,X)$ by $V_1e_n = x_n$ for each $n\in\N$. 

We can now apply \Cref{P:AequivB} to the operator \mbox{$TV_1\in\mathscr{B}(c_0,Y)$}, noting that the unit vector basis $(e_n)_{n\in\N}$ of~$c_0$ is a sequence which satisfies~\eqref{P:AequivB:eq1} because $\inf_{n\in\N}\lVert TV_1e_n\rVert\ge\eta > 1/C$, and that the unit ball of $Y^*\equiv  \bigl(\bigoplus_{\xi\in\Xi} G_\xi^*\bigr)_{\ell_1}$ is weak* sequentially compact, for instance because its elements have countable support, so every sequence is contained in a subspace of the form $\bigl(\bigoplus_{\xi\in\Xi_0} G_\xi^*\bigr)_{\ell_1}$ for some countable subset~$\Xi_0$ of~$\Xi$. It follows that there are operators $U\in\mathscr{B}(Y,c_0)$ and $V_2\in\mathscr{B}(c_0)$ such that $UTV_1V_2 = I_{c_0}$ and $\lVert U\rVert\,\lVert V_2\rVert<C$, so~\eqref{L:Calkin_c0_ell1_FinDim:eq0} is satisfied for the operator $V = V_1V_2$.

In the case where $E=\ell_1$, we no longer have a Rosenthal-style result such as \Cref{P:AequivB}, so we build the factorization from scratch. Set $\eta = \frac{C-1}{3C}\in (0,\frac{1}{3})$. We shall require two elementary facts about operators $S\in\mathscr{B}(X,Y)$ between $\ell_1$-direct sums of finite-dimensional spaces:
\begin{align}
\lVert S\rVert_e &= \inf\bigl\{\lVert(I_Y-R_A)S\rVert : A\in[\Xi]^{<\omega}\bigr\},\label{L:Calkin_c0_ell1_FinDim:eq1}\\
\lVert S\rVert &= \sup\bigl\{ \lVert SJ_\gamma w\rVert : \gamma\in\Gamma,\, w\in B_{F_\gamma}\bigr\}.\label{L:Calkin_c0_ell1_FinDim:eq2}
\end{align}
(In fact, \eqref{L:Calkin_c0_ell1_FinDim:eq1} holds true no matter what form the domain~$X$ of~$S$ has, while \eqref{L:Calkin_c0_ell1_FinDim:eq2} holds true irrespective of the codomain~$Y$.) 
Applying~\eqref{L:Calkin_c0_ell1_FinDim:eq1} to the given operator~$T$ with essential norm~$1$, we can find $A_0\in[\Xi]^{<\omega}$ such that  
\begin{equation}\label{L:Calkin_c0_ell1_FinDim:eq3}
\lVert (I_Y-R_{A_0})T\rVert\le 1+\eta.    
\end{equation}
Now, using~\eqref{L:Calkin_c0_ell1_FinDim:eq2} together with the fact that $\lVert(I_Y-R_A)T(I_X-R_B)\rVert\ge \lVert T\rVert_e>1-\eta$ whenever $A\in[\Xi]^{<\omega}$ and $B\in[\Gamma]^{<\omega}$, 
we can recursively choose a sequence $(A_n)_{n\in\N}$ of dis\-joint sets in~$[\Xi\setminus A_0]^{<\omega}$ and a sequence $(\gamma_n)_{n\in\N}$ of distinct elements in~$\Gamma$ such that, for each $n\in\N$, there is a unit vector $w_n\in F_{\gamma_n}$ for which the vector $y_n = R_{A_n}TJ_{\gamma_n}w_n\in Y$  satisfies 
\begin{equation}\label{L:Calkin_c0_ell1_FinDim:eq4}
\lVert y_n\rVert\ge 1-\eta. \end{equation}
Then $(J_{\gamma_n}w_n)$ is a sequence of disjointly supported unit vectors in~$X$, so we can define a linear isometry $V\in\mathscr{B}(\ell_1,X)$ by $Ve_n = J_{\gamma_n}w_n$ for each $n\in\N$. 

Take $n\in\N$, and choose $f_n\in Y^*$ of norm~$1$ such that $\langle y_n,f_n\rangle = \lVert y_n\rVert$. Since $y_n = R_{A_n}y_n$ and $\lVert R_{A_n}\rVert =1$, we may suppose that $R_{A_n}^*f_n=f_n$. This implies that $\langle y_m,f_n\rangle = 0$ for \mbox{$m\ne n$}, and we can define an operator $U_1\in\mathscr{B}(Y,\ell_1)$ by $U_1y = \bigl(\langle y,f_n\rangle/\lVert y_n\rVert\bigr)_{n\in\N}$; it satisfies $U_1y_n = e_n$ for $n\in\N$, and $\lVert U_1\rVert\le \frac{1}{1-\eta}$ by~\eqref{L:Calkin_c0_ell1_FinDim:eq4}.

The definition of~$y_n$ together with the fact that $A_n\subseteq\Xi\setminus A_0$ means that~$y_n$ is the component of the vector $(I_Y-R_{A_0})TJ_{\gamma_n}w_n$ supported on the set~$A_n$, so by the definition of the $\ell_1$-norm and \eqref{L:Calkin_c0_ell1_FinDim:eq3}--\eqref{L:Calkin_c0_ell1_FinDim:eq4}, we have
\[ \lVert (I_Y-R_{A_0})TJ_{\gamma_n}w_n - y_n\rVert = \lVert (I_Y-R_{A_0})TJ_{\gamma_n}w_n\rVert - \lVert y_n\rVert\le 1+\eta - (1-\eta) = 2\eta.  \]
Consequently the operator $U_2 = U_1(I_Y-R_{A_0})TV\in\mathscr{B}(\ell_1)$ satisfies
\[ \lVert (U_2 -I_{\ell_1})e_n\rVert = \lVert U_1((I_Y-R_{A_0})TJ_{\gamma_n}w_n - y_n)\rVert\le \frac{2\eta}{1-\eta}\qquad (n\in\N). \]
This implies that  $\lVert U_2 - I_{\ell_1}\rVert\le 2\eta/(1-\eta)<1$ by the well-known counterpart of~\eqref{L:Calkin_c0_ell1_FinDim:eq2} for operators defined on the space~$\ell_1$ itself. Therefore $U_2$ is invertible, and 
\[ \lVert U_2^{-1}\rVert\le \Bigl(1 - \frac{2\eta}{1-\eta}\Bigr)^{-1} = \frac{1-\eta}{1-3\eta}. \] 
Set $U = U_2^{-1}U_1(I_Y-R_{A_0})\in\mathscr{B}(Y,\ell_1)$. Then $UTV=I_{\ell_1}$ by the definition of~$U_2$,   
\[ \lVert U\rVert\le \frac{1-\eta}{1 - 3\eta}\frac{1}{1-\eta} = C\qquad\text{and}\qquad \lVert V\rVert =1. \qedhere \]
\end{proof}

\begin{corollary}\label{calkindirectsumincompressible} 
 Let  $X = \bigl(\bigoplus_{n \in \mathbb{N}}\ell_2^n \bigr)_E\oplus E_\Gamma$, where $(E,E_\Gamma) = (c_0,c_0(\Gamma))$ or $(E,E_\Gamma)  = (\ell_1,\ell_1(\Gamma))$ for some index set~$\Gamma$. 
For every non-compact operator $T\in\mathscr{B}(X)$ with $\lVert T\rVert_e=1$ and every constant $C>1$, there are operators $U\in\mathscr{B}(X,E)$ and $V\in\mathscr{B}(E,X)$ such that 
\begin{equation*}
  UTV = I_{E}\qquad\text{and}\qquad \lVert U\rVert\,\lVert V\rVert\le C. 
\end{equation*} 
Consequently the Calkin algebra of~$X$ has the $C$-IFP and is uniformly incompressible.
\end{corollary}

\begin{proof} This follows immediately from \Cref{{L:Calkin_c0_ell1_FinDim}} because we can write $X = \bigl(\bigoplus_{\xi\in\Xi} F_\xi\bigr)_E$, where $\Xi = (\N\times\{1\})\cup(\Gamma\times\{2\})$ and $F_\xi = \ell_2^n$ for $\xi =(n,1)$ and $F_\xi = \K$ for $\xi = (\gamma,2)$.
\end{proof}

\section{The proof of \Cref{mainthm}\ref{mainthm2}}\label{c0KA} 
\noindent Let $[\mathbb{N}]^\omega$ be the set of all infinite subsets of~$\mathbb{N}$, recall that  $[\Gamma]^{<\omega}$ denotes the set of all finite subsets of a general set~$\Gamma$, and consider a family $\mathcal{A}\subseteq[\mathbb{N}]^\omega$ which is \emph{almost disjoint} in the sense that $A\cap B\in[\mathbb{N}]^{<\omega}$ whenever $A,B\in\mathcal{A}$ are distinct. Then the closed subspace~$X_\A$  of~$\ell_\infty$ spanned by the indicator functions $\one_A$ for $A\in\A$ and $\one_{\{n\}}$ for $n\in\mathbb{N}$ is a self-adjoint subalgebra. Hence, assuming complex scalars, we see that~$X_\A$ is isometrically isomorphic to~$C_0(K_\A)$ for  a locally compact Hausdorff space~$K_\A$, which is the \emph{Mr\'{o}wka space} induced by the almost disjoint family~$\A$ (also known as the  Alexandroff--Urysohn space, the $\Psi$-space and the Isbell--Mr\'{o}wka space). Note that~$X_\A$ is non-unital whenever~$\A$ is infinite, so~$K_\A$ is not compact. Furthermore, $K_\A$ is metrizable if and only if $C_0(K_\A)$ is separable, if and only if~$\A$ is countable.

Alternatively, one can define~$K_\A$ explicitly as follows. As a set,
\begin{equation}\label{eq:K_A} K_\A = \{x_n : n \in \mathbb{N} \} \cup \{y_A : A\in \A\}, \end{equation} where $x_n$ is an isolated point for each $n \in \mathbb{N}$, while  a neighbourhood basis of~$y_A$, for each $A \in \A$, is given by the collection of sets of the form $U(A,F) = \{y_A\} \cup \{x \in A \setminus F\}$, where $F \in [\mathbb{N}]^{<\omega}$. It follows easily from this description that~$K_\A$ is \emph{scattered} in the sense that every non-empty subset of it contains a relatively isolated point.  

Our goal is to prove the following theorem.

\begin{theorem}\label{finalCKresult}
  Let $K_\A$ be the Mr\'{o}wka space induced by an uncountable, almost disjoint family $\A\subseteq[\mathbb{N}]^\omega$.
  For  every $C>4$ and every non-compact operator $T\in\mathscr{B}(C_0(K_\A))$ with $\lVert T\rVert_e=1$, there are operators $U\in\mathscr{B}(C_0(K_\A),c_0)$ and $V\in\mathscr{B}(c_0,C_0(K_\A))$ such that 
\begin{equation*}
  UTV = I_{c_0}\qquad\text{and}\qquad \lVert U\rVert\,\lVert V\rVert< C. 
\end{equation*} 
Consequently the Calkin algebra of~$C_0(K_\A)$ has the $C$-IFP and is uniformly incompressible.
\end{theorem}

Before   embarking on the proof of \Cref{finalCKresult}, let us explain how \Cref{mainthm}\ref{mainthm2} follows from it. 

The setting is as follows: for an un\-count\-able, almost disjoint family $\A \subseteq[\mathbb{N}]^\omega$, we say that $C_0(K_\A)$ admits \emph{few operators} if every operator on~$C_0(K_\A)$ is the sum of a scalar multiple of the identity operator and an operator which factors through~$c_0$, so in other words the ideal~$\mathscr{G}_{c_0}(C_0(K_\A))$ of operators factoring through~$c_0$ has codimension~$1$ in~$\mathscr{B}(C_0(K_\A))$. We remark that the ideal~$\mathscr{G}_{c_0}(C_0(K_\A))$ is closed for every Mr\'{o}wka space~$K_\A$ because it coincides with the ideal of operators with separable range, which is a closed operator ideal.

Assuming the Continuum Hypothesis, Koszmider~\cite{K} 
 constructed an un\-count\-able, almost disjoint family $\A \subseteq[\mathbb{N}]^\omega$ for which $C_0(K_\A)$ admits few operators;    a construction within \textsf{ZFC} is given in~\cite{KL}. 
 Kania and Ko\-chanek \cite[Theorem~5.5]{KK} proved that when~$C_0(K_\A)$ admits few operators, $\mcb(C_0(K_\A))$ contains only four closed ideals, namely
\begin{equation}\label{eq:KK}
\{0\}\subsetneq \mck(C_0(K_\A))\subsetneq \mcg_{c_0}(C_0(K_\A)) \subsetneq \mcb(C_0(K_\A)). \end{equation}
This result was also obtained independently by Brooker (unpublished).

In view of~\eqref{Proof:todo} and~\eqref{eq:KK}, to prove \Cref{mainthm}\ref{mainthm2}, we need only establish incompressibility of the quotients $\mcb(C_0(K_\A))/\mcg_{c_0}(C_0(K_\A))\cong\mathbb{K}$ and $\mcb(C_0(K_\A))/\mck(C_0(K_\A))$. The former is finite-dimensional, so trivially has a unique algebra norm, while \Cref{finalCKresult} gives the conclusion for the latter.\smallskip 

It remains to prove \Cref{finalCKresult}. Our strategy  is to  verify that $C_0(K_\A)$ contains a sequence $(f_n)_{n\in\N}$ which satisfies~\eqref{P:AequivB:eq1} and then apply \Cref{P:AequivB}.  We begin by explaining why the latter result applies to the space $Y=C_0(K_\A)$. 

\begin{theorem}\label{weakstardualball:formerpart1} The unit ball of~$C_0(K)^*$ is weak* sequentially compact for every scattered, locally compact Hausdorff  space~$K$.
\end{theorem}

The following well-known observation will be useful in the proof.

\begin{remark}\label{remark:CKhyperplanes:part1}
Let~$K$ be  a scattered, compact Hausdorff space. Then~$C(K)$ contains a complemented subspace which is isomorphic to~$c_0$, so $C(K)$ is isomorphic to its hyperplanes. 

As an application of this fact, suppose that~$K$ is only \emph{locally} compact (as well as scattered and Hausdorff). Then its one-point compactification~$\alpha K$ is scattered. Since $C_0(K)$ is a hyperplane in~$C(\alpha K)$, these two Banach spaces are isomorphic.
\end{remark}

\begin{proof}[Proof of Theorem~{\normalfont{\ref{weakstardualball:formerpart1}}}] In view of \Cref{remark:CKhyperplanes:part1} and the isomorphic nature of this result, we may suppose that~$K$ is compact by replacing it with its one-point com\-pact\-i\-fi\-ca\-tion if necessary. Now the conclusion follows by combining two famous results: on the one hand, Hagler and John\-son \cite[Theorem~1(b)]{HJ} have shown that the unit ball of~$X^*$ is weak* sequentially compact for every Asp\-lund space~$X$, and on the other Namioka and Phelps  \cite[Theorem~18]{NP} have shown that~$C(K)$ is an Asplund space whenever~$K$ is a scattered, compact Hausdorff space. 
\end{proof}

To facilitate the presentation of our proof of \Cref{finalCKresult}, we introduce some more notation. Through\-out the following list, $\A\subseteq [\mathbb{N}]^\omega$ denotes an almost disjoint family.

\begin{itemize}
\item  A (finite or infinite) sequence $(f_n)_{n\in N}$ of scalar-valued functions, all defined on the same set~$K$, is \emph{disjoint} if, for every $k\in K$, there is at most one $n\in N$ such that $f_n(k)\ne 0$.  This is a slight weakening of the usual notion of disjointly supported functions in the case where $K$ is a topological space.
\item For a function $f \in C_0(K_\A)$, define \[    r(f) = \{n\in \mathbb{N} : f(x_n) \neq 0 \}\qquad\text{and}\qquad s(f) = \{A \in \A : f(y_\A) \neq 0\}.\]
  In view of~\eqref{eq:K_A}, these sets correspond to a partition of $\{ k\in K_\A : f(k)\ne 0\}$, and consequently a sequence $(f_n)$ in~$C_0(K_\A)$ is disjoint in the sense of the preceding bullet point if and only if \[r(f_m) \cap r(f_n) = \emptyset\quad\text{and}\quad s(f_m) \cap s(f_n) = \emptyset\quad\text{whenever}\quad m\ne n.\] 
 \item  For $\A'\in [\A]^{<\omega}$, define
\[D(\A') = \bigcup \{ A \cap B : A,B \in \A',\, A \neq B\}\in[\mathbb{N}]^{<\omega}. \] In words, $D(\A')$ is the set of those natural numbers which belong to more than one of the sets in $\A'$. The fact that $D(\A')$ is finite follows from the fact that the family~$\A$ is almost disjoint and the set~$\A'$ is finite.
  \item For $F\in [\mathbb{N}]^{<\omega}$ and $\A'\in[\A]^{<\omega}$, we define a finite-rank operator by
    \[ \Proj(F,\A') = \sum_{n\in F} \delta_{x_n} \otimes \mathbbm{1}_{\{x_n\}} + \sum_{A\in\A'} \delta_{y_A} \otimes \mathbbm{1}_{U(A,F)} \in \mcb(C_0(K_\A)), \]
    where $\delta_k\in C_0(K_\A)^*$ denotes the point evaluation at a point~$k\in K_\A$, that is, \[ \langle f, \delta_k\rangle = f(k)\qquad (f \in C_0(K_\A)); \]
    $\mathbbm{1}_L\colon K_\A\to \{0,1\}$ denotes the indicator function of a subset~$L$ of~$K_\A$, and we recall that the sets~$U(A,F)$ were defined in the text below~\eqref{eq:K_A}; and  as usual,  $\delta_k\otimes\mathbbm{1}_L$ denotes the rank-one operator given by $f\mapsto f(k)\mathbbm{1}_L$. 
\end{itemize} 

The following two lemmas explain the properties of the operators $\operatorname{Proj}(F,\A')$ that we require in our construction. The first is verified by direct computation, so we omit the details.

\begin{lemma}\label{L:projFA}
Let $F\in [\mathbb{N}]^{<\omega}$ and $\A'\in [\A]^{<\omega}$. Then: 
\begin{enumerate}[label={\normalfont{(\roman*)}}]
\item\label{4.4.19i} $\operatorname{Proj}(F,\A')$ is a projection.
\item\label{4.4.19ii} The adjoint $\operatorname{Proj}(F,\A')^*$ of $\operatorname{Proj}(F,\A')$ satisfies
  \begin{equation*}%\label{CKlemmaiii}
    \Proj(F,\A')^*\delta_{y_A} = \begin{cases}  \delta_{y_A} &\quad \text{for}\ A\in \A',\\
    0 &\quad\text{for}\ A\in\A\setminus\A'. \end{cases} \end{equation*}
\item\label{4.4.19iii} Suppose that $D(\A')\subseteq F$. Then $\lVert\Proj(F,\A')\rVert = 1$ (except in the trivial case where~$\A'$ and~$F$ are both empty). 
\item\label{eq:projnorm1} $\lVert I - \Proj(F,\emptyset)\rVert =1$. 
    \end{enumerate}
\end{lemma}

\begin{lemma}\label{CKlemma}
  Let $h\in \operatorname{span}\{\mathbbm{1}_{\{x_n\}},\, \mathbbm{1}_{U(A,\emptyset)} : n \in \mathbb{N},\, A \in \A\}$, where $\A\subseteq [\mathbb{N}]^\omega$ is an almost dis\-joint family. Then:
  \begin{enumerate}[label={\normalfont{(\roman*)}}]
    \item\label{CKlemmai} There exists a set $H\in [\mathbb{N}]^{<\omega}$ such that
\begin{equation*}
D(s(h)) \subseteq H\qquad\text{and}\qquad h \in \operatorname{span}\{\mathbbm{1}_{\{x_n\}},\, \mathbbm{1}_{U(A,H)} : n\in H,\, A \in s(h)\}.\end{equation*}
\item\label{CKlemmaii} Suppose that $H\in[\mathbb{N}]^{<\omega}$ is chosen as in~\ref{CKlemmai}. Then, for every $G\in[\mathbb{N}]^{<\omega}$ with $H\subseteq G$, the function $g = (I-\Proj(G,\emptyset))h\in C_0(K_\A)$ satisfies
  \begin{equation*} s(g)=s(h)\qquad\text{and}\qquad g=\sum_{A \in s(g)} g(y_A) \mathbbm{1}_{U(A,G)}. \end{equation*}
\end{enumerate}  
\end{lemma}

\begin{proof}
  \ref{CKlemmai}. Since $s(h)$ is a finite subset of~$\A$, we may express~$h$ in the form 
  \begin{equation*}%\label{gasaspan} 
  h = \sum_{n\in N} \lambda_n\mathbbm{1}_{\{x_n\}} + \sum_{A\in s(h)} h(y_A) \mathbbm{1}_{U(A,\emptyset)} \end{equation*} for some $N\in[\mathbb{N}]^{<\omega}$ and some scalars~$\lambda_n$ for  $n\in N$. Then the set $H = D(s(h))\cup N\in[\N]^{<\omega}$ satisfies \[h = \sum_{n\in H} h(x_n) \mathbbm{1}_{\{x_n\}} + \sum_{A \in s(h)}h(y_A)\mathbbm{1}_{U(A,H)}.  \]
  
\ref{CKlemmaii}. First, we note that  $g(y_A)=h(y_A)$ for every $A \in \A$ because $s(\Proj(G,\emptyset)h) = \emptyset$, so $s(g)=s(h)$. To help the readability of the proof of the second part, we define \[ f = \sum_{A \in s(g)} g(y_A) \mathbbm{1}_{U(A,G)}. \] Then, what we seek to verify is that $g=f$.  It is clear by inspection that $g(y_A)=f(y_A)$ for every $A \in \A$. So, it remains to check that $g(x_n)=f(x_n)$ for every $n\in \mathbb{N}$. There are three cases to consider: 
\begin{itemize}
\item If $n\in G$, then $ f(x_n)=0$ and $(\Proj(G,\emptyset)h)(x_n) = h(x_n)$, so $g(x_n)=0$ as well.
\item If $n\in A\setminus G$ for some $A\in s(g)$, then $n\notin B$ for every $B\in s(g)\setminus\{A\}$ because $A\cap B\subseteq D(s(g))=D(s(h))\subseteq H\subseteq G$.
Therefore we have \[ f(x_n)=g(y_A)=h(y_A)=h(x_n) = g(x_n). \]
\item The final case is that $n\notin G$ and $n\notin A$ for every $A\in s(g)$. This implies that $f(x_n)=0$, $h(x_n)=0$  and $(\Proj(G,\emptyset)h)(x_n) = 0$, so we also have $g(x_n)=0$. \qedhere
 \end{itemize}
\end{proof}

\begin{proof}[Proof of Theorem~{\normalfont{\ref{finalCKresult}}}] As previously mentioned, our strategy is to show that, for every \mbox{$C>4$} and every non-compact operator $T \in \mcb(C_0(K_\A))$ with $\lVert T\rVert_e = 1$, $C_0(K_\A)$ contains a  sequence $(f_n)_{n\in\N}$ which satisfies condition~\eqref{P:AequivB:eq1} in \Cref{P:AequivB}. This will prove the result because \Cref{weakstardualball:formerpart1} ensures that \Cref{P:AequivB} applies.

Take $\delta\in (\frac1C, \frac{1}{4})$.  By recursion, we shall   construct a disjoint sequence $(f_n)_{n\in\N}$ in the unit ball of~${C_0(K_\A)}$ such that \begin{equation}\label{eq:lowerboundTf} \lVert Tf_n\rVert\geq \delta\qquad (n\in\mathbb{N}).
\end{equation}
We observe that a sequence $(f_n)$ with these properties will satisfy~\eqref{P:AequivB:eq1} because, for each $k\in K_\A$, there is at most one $n\in\N$ such that $f_n(k)\ne 0$, and therefore $\sum_{n=1}^\infty \lvert f_n(k)\rvert\le1$. 
We carry out the recursive construction in  two distinct cases, depending on the norms of the restrictions of a certain family of operators to the subspace $X_{00} = \operatorname{span}\{\mathbbm{1}_{\{x_n\}}: n \in \mathbb{N}\}$ of~$C_0(K_\A)$.\medskip

\textbf{Case 1.} Suppose that
  \begin{equation}\label{eq:1star} \lVert T(I-\Proj(F,\emptyset))|_{X_{00}}\rVert > \delta\qquad  (F\in [\mathbb{N}]^{<\omega}). \end{equation}
  In this case, it is quite easy to construct a disjoint sequence $(f_n)_{n\in\N}$ satisfying~\eqref{eq:lowerboundTf} inside the unit ball of~$X_{00}$.

 Indeed, to start the recursion, we apply~\eqref{eq:1star} with $F_0=\emptyset$ to see that $\|T|_{X_{00}}\| >\delta$, so we can select $f_1\in B_{X_{00}}$ such that $\|Tf_1\|> \delta$.

  Now assume recursively that, for some $n\in\mathbb{N}$, we have chosen a disjoint sequence of functions $(f_j)_{j=1}^n$ in~$B_{X_{00}}$ for which $\|Tf_j\| >\delta $ for each $j \in \{1,\dots,n\}$. Set $F_n = \bigcup^n_{j=1} r(f_j)$, and  observe that $F_n\in[\mathbb{N}]^{<\omega}$ because $f_1,\ldots,f_n\in X_{00}$. Hence we may apply~\eqref{eq:1star} to find a func\-tion $g_{n+1}\in B_{X_{00}}$ for which \[ \lVert T(I-\Proj (F_n,\emptyset))g_{n+1}\rVert > \delta.\] 
  Define $f_{n+1} =(I-\Proj(F_n, \emptyset))g_{n+1}\in X_{00}$. Then we have $\lVert Tf_{n+1}\rVert > \delta$, $r(f_{n+1})\subseteq\mathbb{N}\setminus F_n$, and $\|f_{n+1}\| \leq \|g_{n+1}\| \leq 1$ by \Cref{L:projFA}\ref{eq:projnorm1}, so the recursion continues.\medskip 

\textbf{Case 2.} Suppose that~\eqref{eq:1star}  fails, so that $\lVert T(I-\Proj(F,\emptyset))|_{X_{00}}\rVert\leq \delta$ for some $F\in [\mathbb{N}]^{<\omega}$, and set $S= T(I-\Proj(F,\emptyset))\in\mcb(C_0(K_\A))$. Since $\Proj(F,\emptyset)$ is a finite-rank operator, we have $\|S\|_e = \|T\|_e = 1$.
In this case, we shall construct the disjoint sequence $(f_n)_{n\in\N}$ satisfying~\eqref{eq:lowerboundTf} within  the unit ball of the subspace
\[ Y_{00} = \operatorname{span}\{\mathbbm{1}_{\{x_n\}},\, \mathbbm{1}_{U(A,\emptyset)} : n \in \mathbb{N},\, A \in \A\}, \]
which is dense in~$C_0(K_\A)$.  This  argument is somewhat more involved, with the main part consisting of a recursive construction of a disjoint sequence of functions $(g_n)_{n\in \mathbb{N}}$ in~$B_{Y_{00}}$ and a sequence $(G_n)_{n \in \mathbb{N}}$ in~$[\mathbb{N}]^{<\omega}$ for which 
\begin{equation}\label{star} \|Sg_n\| > \delta,\quad\
    D\Bigl(\bigcup_{j=1}^n s(g_j)\Bigr)\subseteq G_n\quad\ \text{and}\quad\
    g_n = \sum_{A \in s(g_n)} g_n(y_A)\mathbbm{1}_{U(A,G_n)}\qquad (n \in \mathbb{N}).
\end{equation}

To begin this recursion, observe that there is a function $h_1 \in B_{Y_{00}}$ with $\|Sh_1\| > 2\delta$ because $\|S\|\geq \|S\|_e = 1 > 2\delta$ and~$Y_{00}$ is dense in $C_0(K_\A)$. Applying \Cref{CKlemma}\ref{CKlemmai}, we can find  a set  $H_1 \in [\mathbb{N}]^{<\omega}$ such that
\[ D(s(h_1))\subseteq H_1\qquad\text{and}\qquad h_1 \in \operatorname{span}\{\mathbbm{1}_{\{x_m\}},\, \mathbbm{1}_{U(A,H_1)} : m\in H_1,\, A \in s(h_1)\}. \] Define $G_1 = H_1$, $k_1 = \Proj(H_1,\emptyset)h_1\in B_{X_{00}}$ and $g_1 = h_1-k_1 = (I-\Proj(H_1,\emptyset))h_1 \in B_{Y_{00}}$. Then  $\|Sk_1\|\leq \delta$ by the choice of~$S$, so that \[ \|Sg_1\| \geq \|Sh_1\| - \|Sk_1\| > 2\delta - \delta = \delta,  \]
while the second and third identities in~\eqref{star} both follow from \Cref{CKlemma}\ref{CKlemmaii}. This completes the base step.

Now assume that, for some $n\in\N$, we have selected disjointly supported functions $g_1,\ldots,g_n \in B_{Y_{00}}$ and sets $G_1,\ldots,G_n\in[\mathbb{N}]^{<\omega}$ satisfying~\eqref{star}. The definition of~$Y_{00}$ implies that the set $\A_n = \bigcup_{j=1}^n s(g_j)$ is finite, so $Q_n = \Proj(G_n,\A_n)$ is a finite-rank projection, and $\lVert S(I-Q_n)\rVert\geq \lVert S\rVert_e = 1 > 4\delta$.  Hence  we can find a func\-tion $\zeta_{n+1} \in B_{Y_{00}}$ for which $\|S(I-Q_n)\zeta_{n+1}\| > 4\delta$.

By hypothesis, $D(\A_n)\subseteq G_n$, so \Cref{L:projFA}\ref{4.4.19iii} implies that $\lVert Q_n\rVert = 1$. Consequently the function $h_{n+1} = \frac{1}{2}(I-Q_n)\zeta_{n+1}$ belongs to~$B_{Y_{00}}$. With another application of \Cref{CKlemma}\ref{CKlemmai}, we produce a set $H_{n+1}\in [\mathbb{N}]^{<\omega}$ for which $D(s(h_{n+1}))\subseteq H_{n+1}$ and
\[ h_{n+1} \in \operatorname{span}\{\mathbbm{1}_{\{x_m\}},\,\mathbbm{1}_{U(A,H_{n+1})} : m\in H_{n+1},\, A\in s(h_{n+1})\}. \]
For every $A\in\A_n$, we have 
\[ h_{n+1}(y_A) = \frac{1}{2}\langle \zeta_{n+1},(I-Q_n)^*\delta_{y_A}\rangle = 0 \]
because  $Q_n^*\delta_{y_A} = \delta_{y_A}$ by \Cref{L:projFA}\ref{4.4.19ii}, and therefore
 $s(h_{n+1})\cap\A_n=\emptyset$.
Define
\begin{gather*}
G_{n+1} = H_{n+1} \cup D(\A_n\cup s(h_{n+1}))\in [\mathbb{N}]^{<\omega},\qquad
k_{n+1} = \Proj(G_{n+1},\emptyset)h_{n+1}\in B_{X_{00}},\\ g_{n+1} = h_{n+1}-k_{n+1} = (I-\operatorname{Proj}(G_{n+1},\emptyset))h_{n+1} \in B_{Y_{00}}. \end{gather*}
Then  we have $\lVert Sk_{n+1}\rVert\leq \delta$, which implies that 
\[ \lVert Sg_{n+1}\rVert\geq\lVert Sh_{n+1}\rVert  - \delta = \frac{1}{2}\lVert S(I-Q_n)\zeta_{n+1}\rVert -\delta > \frac{4\delta}{2}-\delta = \delta. \]
Furthermore, \Cref{CKlemma}\ref{CKlemmaii} shows that
\begin{equation}\label{Eq:18july}
  s(g_{n+1}) = s(h_{n+1})\qquad\text{and}\qquad g_{n+1} = \sum_{A \in s(g_{n+1})} g_{n+1}(y_A)\mathbbm{1}_{U(A,G_{n+1})}. \end{equation}
In particular,  we see that $D\bigl(\bigcup^{n+1}_{j=1}s(g_j)\bigr)\subseteq G_{n+1}$, bearing the definition of~$\A_n$ in mind. Thus~\eqref{star} is satisfied for $n+1$.

It remains to verify that $g_{n+1}$ is disjoint from $g_j$ for each $j\in\{1,\ldots,n\}$. Since \[ s(g_{n+1})\cap s(g_j)\subseteq s(h_{n+1})\cap\A_n=\emptyset, \] we need only check that $r(g_{n+1})\cap r(g_j) = \emptyset$. Take $m\in r(g_{n+1})$, so that $m\in A\setminus G_{n+1}$ for some $A\in s(g_{n+1})$ by~\eqref{Eq:18july}. Then, for each $B \in s(g_j)$, we see that $m \notin B$ because \[A\cap B \subseteq D\biggl(\bigcup^{n+1}_{i=1}s(g_i)\biggr)\subseteq G_{n+1}. \] Hence the final identity in~\eqref{star} for the function $g_j$ implies that
\[ g_j(x_m) = \sum_{B\in s(g_j)} g_j(y_B)\mathbbm{1}_{U(B,G_j)}(x_m) = 0, \] and therefore $m\notin r(g_j)$. This completes the recursive construction of the sequences $(g_n)_{n\in \mathbb{N}}$ and $(G_n)_{n\in \mathbb{N}}$.

We conclude the proof by defining $f_n = (I-\Proj(F,\emptyset))g_n$ for every $n\in \mathbb{N}$. Then $\lVert f_n\rVert\leq \lVert g_n\rVert\leq 1$ by \Cref{L:projFA}\ref{eq:projnorm1}, $\lVert Tf_n\rVert = \lVert Sg_n\rVert > \delta$ by the left-hand part of~\eqref{star}, and the sequence $(f_n)_{n \in \mathbb{N}}$ is disjoint because $s(f_n)= s(g_n)$ and $r(f_n)\subseteq r(g_n)$ for every $n \in \mathbb{N}$, where the sequence $(g_n)_{n \in \mathbb{N}}$ is disjoint. \end{proof}

\section{Incompressibility of the Calkin algebra of $C_0(K)$ for scattered $K$}\label{section5}
\noindent 
In each of the cases we encountered in the proof of \Cref{mainthm}, we deduced incompressibility from quantitative factorizations of idempotent operators, which in view of \Cref{ex:nonuniformincompress} would appear to be a much stronger property than necessary. Notably, in the previous section we used \Cref{P:AequivB} to show that  the Calkin algebra of~$C_0(K_\A)$ is (uniformly) incompressible for every uncountable, almost disjoint family $\A\subseteq[\N]^\omega$.

The aim of this section is to prove that for a scattered, locally  compact Hausdorff space~$K$, the idempotent factorization property is not only a sufficient, but also a necessary condition for incompressibility of the Calkin algebra of~$C_0(K)$. 
The significance of this result is that when somebody seeks to determine whether the Calkin algebra of~$C_0(K)$ is incompressible, they will know that verifying condition~\ref{charCKincompress1} below (as we did in the previous section for $K=K_\A$) is most likely the optimal route. We consider the case where~$K$ equals the ordinal interval~$[0,\omega^\omega)$ particularly interesting. 

Recall from~\eqref{Eq:essentialnorm} that~\mbox{$\lVert\,\cdot\,\rVert_e$} denotes the essential norm.

  \begin{theorem}\label{charCKincompress}     The following conditions are equivalent for  a scattered, locally compact Haus\-dorff space~$K\colon$
\begin{enumerate}[label={\normalfont{(\alph*)}}]
\item\label{charCKincompress1} There is a constant $\delta\in(0,1)$ such that, for every non-compact operator $T\in\mathscr{B}(C_0(K))$ with $\lVert T\rVert_e =1$, $C_0(K)$~con\-tains a sequence $(f_n)_{n\in\N}$ with  
\begin{equation}\label{charCKincompress:eq1} 
    \sup_{k\in K} \sum_{n=1}^\infty\lvert f_n(k)\rvert\le 1\qquad\text{and}\qquad \inf_{n\in\N}\lVert Tf_n\rVert >\delta.
\end{equation} 
\item\label{charCKincompress2} There is a constant $C>1$ such that, for every non-compact operator $T\in\mathscr{B}(C_0(K))$ with $\lVert T\rVert_e =1$, there are operators $U\in \mcb(C_0(K),c_0)$ and $V\in\mathscr{B}(c_0,C_0(K))$ with $\lVert U\rVert\,\lVert V\rVert<C$ such that  $UTV=I_{c_0}$. 
\item\label{charCKincompress3} The Calkin algebra $\mathscr{B}(C_0(K))/\mathscr{K}(C_0(K))$ is uniformly incompressible. 
\item\label{charCKincompress4} The Calkin algebra $\mathscr{B}(C_0(K))/\mathscr{K}(C_0(K))$ is incompressible. 
\item\label{charCKincompress5} The  algebra norm~$\nu$ given by~\eqref{eq:def_nu} below dominates the essential norm in the sense that there is a constant $\eta>0$ such that
  \[ \nu(T+\mathscr{K}(C_0(K)))\ge\eta\lVert T\rVert_e\qquad (T\in\mathscr{B}(C_0(K))). \] 
\end{enumerate}
\end{theorem}

The proof of \Cref{charCKincompress} has two main ingredients. The first is \Cref{P:AequivB}, which was  already established in \Cref{Section3} and has  the equivalence of conditions~\ref{charCKincompress1} and~\ref{charCKincompress2} as an immediate consequence. The other is the algebra norm~$\nu$ on the Calkin algebra of~$C_0(K)$ that appears in condition~\ref{charCKincompress5}.

To define it, we note that according to the proof of \cite[Prop\-o\-si\-tion~5.4(ii)]{KK}, the identity operator on~$c_0$ factors through every non-compact ope\-ra\-tor on~$C_0(K)$ whenever~$K$ is a scattered, locally compact Hausdorff space. (This will also follow from our results, as we shall explain in detail at the beginning of the proof of \Cref{P:JPSalgnorm} below.) Therefore we can define  a map $\nu\colon \mathscr{B}(C_0(K))/\mathscr{K}(C_0(K))\to[0,\infty)$ by 
\begin{equation}\label{eq:def_nu} \nu(T+\mathscr{K}(C_0(K))) = \begin{cases} 0\quad &\text{if}\ T\in\mathscr{K}(C_0(K)),\\ \sup\biggl\{ \displaystyle{\frac{1}{\lVert U\rVert\, \lVert V\rVert}} : I_{c_0} - UTV\in\mathscr{K}(c_0)\biggr\}\  &\text{otherwise.}
\end{cases} \end{equation}

As already indicated, our aim is to establish the following result. 

\begin{proposition}\label{P:JPSalgnorm}
Let~$K$ be a scattered, locally compact Hausdorff space. Then~$\nu$ defined by~\eqref{eq:def_nu} is an algebra norm on~$\mathscr{B}(C_0(K))/\mathscr{K}(C_0(K))$, and 
\begin{equation}\label{P:JPSalgnorm:eq1}
    \nu(T+\mathscr{K}(C_0(K))\le \lVert T\rVert_e\qquad (T\in\mathscr{B}(C_0(K))).
\end{equation} 
\end{proposition}

\begin{remark}\label{remark:CKhyperplanes:part2} Recall from \Cref{remark:CKhyperplanes:part1} that~$C_0(K)$ is isomorphic to~$C(\alpha K)$ whenever~$K$ is a scattered, locally compact Hausdorff space and~$\alpha K$ denotes its one-point compactification, which is also scattered.
Consequently, in view of the of the isomorphic invariance of \Cref{charCKincompress}, we could have stated it only for compact~$K$ without formally losing any generality. We have nevertheless chosen to state it in the locally compact case for the following reasons: 
\begin{enumerate}[label={\normalfont{(\roman*)}}]
\item The isomorphism between~$C_0(K)$ and~$C(\alpha K)$ is not isometric (as the example of~$c_0$ and~$c$ illustrates; see~\cite{cambern} for details), and the constants~$\delta$, $C$ and~$\eta$ in conditions~\ref{charCKincompress1}, \ref{charCKincompress2} and~\ref{charCKincompress5} may therefore change when passing from one space to the other. We consider the values of these constants to be of some interest in the context of quantitative factorizations.
\item The present version  includes the case $K=\N$ (or in other words $C_0(K)= c_0$) explicitly, which seems natural given the role~$c_0$ plays in the result. Perhaps more importantly, when proving \Cref{charCKincompress} (specifically, in the proof of \Cref{P:JPSalgnorm}), we shall apply various lemmas to~$c_0$ as well as~$C_0(K)$. If the latter class did not contain~$c_0$, some statements would become more complicated. 
\item   When we  applied a precursor of \Cref{charCKincompress} in \Cref{c0KA}, it was for a locally compact space; the ``vanishing at infinity'' property was very helpful in that proof.
\end{enumerate}
\end{remark}

Before initiating the preparations for the proof of \Cref{P:JPSalgnorm}, let us point out that the definition of~$\nu$ and the proof that it is an algebra norm are  heavily inspired by a similar result in~\cite{J-IHABB} concerning operators that the identity operator on~$\ell_2$ factors through. Our version, which involves replacing~$\ell_2$ with~$c_0$, is somewhat harder to prove because~$c_0$ has a much richer subspace structure than~$\ell_2$.

We begin with a variant  of a standard result  about non-compact operators defined on~$c_0$, tailored to our particular needs. 
Combining it with \Cref{L:BA195}, we  arrive at  \Cref{L:factorizationdiagram}, which we shall apply in a number of slightly different contexts in the proof of \Cref{P:JPSalgnorm}. To ensure that it covers all these applications, we have stated the exact details we require in each case, even though this means there is some redundancy in the  statement.

\begin{lemma}\label{L:c0iso} Let $R\in\mathscr{B}(X,Y)$ be a non-compact operator between Banach spaces~$X$ and~$Y$, where~$X$ is $C$-isomorphic to~$c_0$ for some constant $C\ge 1$. Then, for each $\eta\in(0,1)$, $X$ contains a $C$-complemented subspace~$W$ which is $C$-isomorphic to~$c_0$ and satisfies
\begin{equation}\label{L:c0iso:eq1} 
    (1+\eta)C\lVert R\rVert_e\lVert w\rVert\ge \lVert Rw\rVert\ge \frac{1-\eta}{C}\lVert R\rVert_e\lVert w\rVert\qquad (w\in W).
\end{equation}     
\end{lemma}

\begin{proof}
    Take an isomorphism $U\in\mathscr{B}(c_0,X)$ with $\lVert U\rVert\,\lVert U^{-1}\rVert\le C$, and set $S = RU\in\mathscr{B}(c_0,Y)$. Choose $T\in\mathscr{K}(c_0,Y)$ such that $\lVert S+T\rVert\le (1+\frac{\eta}2)\lVert S\rVert_e$. We have $TP_n\to T$ as $n\to\infty$ because the basis of~$c_0$ is shrinking, so $\lVert T(I_{c_0}-P_{m_0})\rVert\le\frac{\eta}2\lVert S\rVert_e$ for some $m_0\in\N$. 
    
    On the other hand,  $\lVert S(I_{c_0}-P_m)\rVert>(1-\frac{\eta}2)\lVert S\rVert_e$ for each $m\in\N$, and therefore we can find an integer $m'>m$ and a unit vector $w\in\operatorname{span}\{ e_j : m< j\le m'\}$ such that $\lVert Sw\rVert>(1-\frac{\eta}2)\lVert S\rVert_e$. Using this, we can recursively choose integers $m_1<m_2<\cdots$ with $m_1>m_0$ and unit vectors $w_n\in\operatorname{span}\{ e_j : m_{n-1}<j\le m_n\}$ such that \[ \lVert Sw_n\rVert>\Bigl(1-\frac{\eta}2\Bigr)\lVert S\rVert_e\qquad (n\in\N). \] 
    
    Since the linear map on~$c_0$ determined by $e_n\mapsto w_n$ for each $n\in\N$ is an isometry, \Cref{Rosenthal} implies that we can find a subsequence $(w_{n_j})$ of $(w_n)$ such that the restriction of~$S$ to the subspace $W_0 = \overline{\operatorname{span}}\{ w_{n_j} : j\in\N\}$ is bounded below by~$(1-\eta)\lVert S\rVert_e$. Furthermore,   $(w_{n_j})$ is  a normalized block basic sequence of~$(e_n)$, so~$W_0$ is isometrically isomorphic to~$c_0$, and we can take a projection~$Q$ of~$c_0$ onto~$W_0$ with $\lVert Q\rVert=1$. Consequently the subspace $W = U[W_0]$ is $C$-isomorphic to~$c_0$ and complemented in~$X$ via the projection $UQU^{-1}$, which has norm at most~$C$. 
    
    Finally, to verify~\eqref{L:c0iso:eq1}, take $w\in W$, and set $w_0 = U^{-1}w\in W_0$. Then $Rw = Sw_0$ and $w_0 = (I_{c_0} - P_{m_0})w_0$, so 
    \begin{align*} \lVert Rw\rVert &\le \lVert (S+T)w_0\rVert + \lVert T(I_{c_0} - P_{m_0})w_0\rVert\le (1+\eta)\lVert S\rVert_e\lVert w_0\rVert\\ &\le (1+\eta)\lVert R\rVert_e\lVert U\rVert\,\lVert U^{-1}\rVert\,\lVert w\rVert\le (1+\eta)C\lVert R\rVert_e\lVert w\rVert  \end{align*}
    and 
    \begin{align*}
       \lVert Rw\rVert &\ge (1-\eta)\lVert S\rVert_e \lVert w_0\rVert\ge (1-\eta)\frac{\lVert SU^{-1}\rVert_e}{\lVert U^{-1}\rVert}\frac{\lVert Uw_0\rVert}{\lVert U\rVert}\ge\frac{1-\eta}{C} \lVert R\rVert_e\lVert w\rVert. \qedhere  
    \end{align*}
\end{proof}

\begin{definition}
    Let $X$, $Y$ and~$Z$ be Banach spaces. An operator $R\in\mathscr{B}(X,Y)$ \emph{fixes a copy of~$Z$} if there is an operator $V\in\mathscr{B}(Z,X)$ such that the composite operator~$RV$ is bounded below. 
\end{definition}

\begin{corollary}\label{L:factorizationdiagram}
Let~$X$ and~$Y$ be Banach spaces for which the unit ball of~$Y^*$ is weak* se\-quentially compact, and let $R\in\mathscr{B}(X,Y)$ be an operator which fixes a copy of~$c_0$. Then, for every $\epsilon\in(0,1)$, there exist a constant $\xi>0$ and a closed subspace~$W$ of~$X$ such that~$W$ is isomorphic to~$c_0$, 
\begin{equation}\label{L:180423unify:eq1}
      (1+\epsilon)\xi\lVert w\rVert\ge \lVert Rw\rVert\ge (1-\epsilon)\xi\lVert w\rVert\qquad (w\in W), 
\end{equation}  
the subspace $R[W]$ is $(1+\epsilon)$-complemented in~$Y$ and $(1+\epsilon)$-isomorphic to~$c_0$, and  there are operators $U\in\mathscr{B}(Y,c_0)$ and $V\in\mathscr{B}(c_0,X)$ such that 
  \begin{equation}\label{R:IDc0Factorization:eq2} 
 URV = I_{c_0}\qquad\text{and}\qquad  \lVert U\rVert\,\lVert V\rVert\le \frac{(1+\epsilon)^2}{(1-\epsilon)\xi}.
    \end{equation}    
\end{corollary}

\begin{proof} Given $\epsilon\in(0,1)$,  choose $\eta>0$ such that $\eta^2+2\eta\le\epsilon$. Since~$R$ fixes a copy of~$c_0$, we can take a closed subspace~$X_1$ of~$X$ such that~$X_1$ is isomorphic to~$c_0$ and  the restriction of~$R$ to~$X_1$ is bounded below. By James' Distortion Theorem~\cite{James}, $X_1$ contains a closed subspace~$X_2$ which is $(1+\eta)$-iso\-mor\-phic to~$c_0$. The restriction of~$R$ to~$X_2$ is non-compact because it is bounded below, so in view of \Cref{L:c0iso}, we can find a closed subspace~$X_3$ of~$X_2$ such that~$X_3$ is isomorphic to~$c_0$ and 
    \begin{equation}\label{L:keystep:eq3}
      (1+\eta)^2\lVert R|_{X_2}\rVert_e\lVert w\rVert\ge\lVert Rw\rVert\ge\frac{1-\eta}{1+\eta}\lVert R|_{X_2}\rVert_e\lVert w\rVert \qquad (w\in X_3).
    \end{equation}
    
    Now, applying \Cref{L:BA195} to the restriction of~$R$ to~$X_3$ and with the constant $C=1+\epsilon$, we obtain a closed subspace~$Z$ of~$R[X_3]$ such that~$Z$ is $(1+\epsilon)$-com\-ple\-mented in~$Y$ and $(1+\epsilon)$-isomorphic to~$c_0$. Set $W = R^{-1}[Z]\cap X_3$ and $\xi = \lVert R|_{X_2}\rVert_e$. Then~$R$ maps~$W$ isomorphically onto~$Z$, so in particular~$W$ is isomorphic to~$c_0$, and~\eqref{L:180423unify:eq1} follows from~\eqref{L:keystep:eq3} because the choice of~$\eta$ implies that $1+\epsilon\ge(1+\eta)^2$ and $1-\epsilon\le (1-\eta)/(1+\eta)$.  Finally the existence of the operators~$U$ and~$V$ satisfying~\eqref{R:IDc0Factorization:eq2} also follows from \Cref{L:BA195} because the restriction of~$R$ to~$X_3$ is bounded below by~$(1-\epsilon)\xi$.     
\end{proof}

When applying this result to~$C_0(K)$ for $K$ scattered, the following variant of a classical theorem of Pe\l{}\-czy\'{n}\-ski will be helpful.

\begin{theorem}\label{weakstardualball}
Let $K$ be a scattered, locally compact Hausdorff  space. An operator from $C_0(K)$ into a Banach space is compact if and only if it is weakly compact, if and only if it does not fix a copy of~$c_0$. 
\end{theorem}

\begin{proof} In view of \Cref{remark:CKhyperplanes:part1} and the isomorphic nature of this result, we may suppose that~$K$ is compact by replacing it with its one-point com\-pact\-i\-fi\-ca\-tion if necessary. 
Pe\l{}\-czy\'{n}\-ski~\cite{pelczynski} showed that an operator defined on a $C(K)$-space is weakly compact if and only if it does not fix a copy of~$c_0$, and of course every compact operator is weakly compact. 
     
To complete the proof, suppose that~$R$ is a weakly compact operator defined on~$C(K)$. By Gantmacher's Theorem, its adjoint~$R^*$ is also weakly compact. This implies that~$R^*$ is compact because a famous result of Rudin~\cite{rudin} states that its codomain~$C(K)^*$ is isomorphic to~$\ell_1(K)$, which has the Schur property. Hence~$R$ is compact by Schauder's Theorem. 
\end{proof}

Finally, before proving \Cref{P:JPSalgnorm}, we establish an elementary lemma that will help us shorten a couple of steps in the proof.

\begin{lemma}\label{L:HNA}
    Let $X$, $Y$ and $Z$ be infinite-dimensional Banach spaces, and suppose that the operators  $R\in\mathscr{B}(X,Y)$, $T\in\mathscr{B}(Y,Z)$ and $U\in\mathscr{B}(Z,X)$ satisfy
    \begin{equation}\label{L:HNA:eq1}
        I_X - UTR\in\mathscr{K}(X).
    \end{equation}
    Then~$R[X]$ is closed and isomorphic to a closed subspace of finite codimension in~$X$, and the restriction of~$T$ to any closed, infinite-dimensional subspace of~$R[X]$ is non-compact.
\end{lemma}

\begin{proof} 
Equation~\eqref{L:HNA:eq1} implies that~$UTR$ is a Fredholm operator. Therefore~$R$ is an upper semi-Fredholm operator by \cite[Corollary~1.3.4]{CPY}, so~$R$ has closed range and finite-dimensional kernel. Let~$W$ be a closed, complementary subspace of~$\ker R$; that is, $W+\ker R = X$ and $W\cap \ker R=\{0\}$.  Then the restriction of~$R$ to~$W$ is an isomorphism onto~$R[X]$. 

Suppose that~$Y_0$ is a closed subspace of~$R[X]$ such that $T|_{Y_0}$ is compact, set $X_0 = R^{-1}[Y_0]$, and regard the restriction~$\widetilde{R}$ of~$R$ to~$X_0$ as an operator into~$Y_0$; it satisfies $UTR|_{X_0} = UT|_{Y_0}\widetilde{R}$, which implies that $UTR|_{X_0}$ is compact. Hence~$X_0$ is finite-dimensional by~\eqref{L:HNA:eq1}, so $R[X_0]=Y_0$ is also finite-dimensional.
\end{proof}

\begin{remark} The same proof as above will work if one replaces $\mathscr{K}(X)$ with the larger ideal of inessential operators in~\eqref{L:HNA:eq1}.
\end{remark}

\begin{proof}[Proof of Proposition~{\normalfont{\ref{P:JPSalgnorm}}}]
  We shall apply \Cref{L:factorizationdiagram} several times in this proof, in each case to a non-compact operator~$R$ whose domain and codomain are (isomorphic to) either~$c_0$ or~$C_0(K)$. To avoid repetition, let us once and for all state that Theorems~\ref{weakstardualball} and~\ref{weakstardualball:formerpart1}  
  ensure that the hypotheses of \Cref{L:factorizationdiagram} on~$R$ and its codomain are satisfied in this case. 

This observation shows in particular that the identity operator on~$c_0$ factors through every non-compact operator on~$C_0(K)$, so the definition of~$\nu$ makes sense. Furthermore, it is clear  that~$\nu$ is faithful and absolutely homogeneous. 

In the remainder of the proof, suppose that $T_1,T_2\in\mathscr{B}(C_0(K))$.\smallskip

\textbf{Subadditivity.} Let us begin by observing that the inequality \[ \nu(T_1+T_2+\mathscr{K}(C_0(K)))\le\nu(T_1+\mathscr{K}(C_0(K)))+\nu(T_2+\mathscr{K}(C_0(K))) \]
 is trivial if $T_1+T_2$ is compact. Otherwise, for each $\epsilon\in(0,1)$, we can find operators $U\in\mathscr{B}(C_0(K),c_0)$ and $V\in\mathscr{B}(c_0,C_0(K))$ such that 
\begin{equation}\label{JPSsubmult:eq4}
 I_{c_0} - U(T_1+T_2)V\in\mathscr{K}(c_0)   
\end{equation}
and 
\begin{equation*}%\label{JPSsubmult:eq5}
\frac1{\lVert U\rVert\,\lVert V\rVert}\ge (1-\epsilon)\nu(T_1+T_2+\mathscr{K}(C_0(K))). 
\end{equation*}
Set $R_i = UT_iV\in\mathscr{B}(c_0)$ for $i=1,2$. By~\eqref{JPSsubmult:eq4}, at least one of these operators is non-compact, say~$R_1$ (relabelling them if necessary). 
Therefore we can apply \Cref{L:factorizationdiagram} to find a constant $\xi_1>0$ and a closed subspace~$W_1$ of~$c_0$ such that~$W_1$ is isomorphic to~$c_0$, 
\begin{equation}\label{JPSsubmult:eq6}
      (1+\epsilon)\xi_1\lVert w\rVert\ge \lVert R_1w\rVert\ge (1-\epsilon)\xi_1\lVert w\rVert\qquad (w\in W_1) 
\end{equation}   
and  $I_{c_0} = U_1R_1V_1$ for  some operators $U_1,V_1\in\mathscr{B}(c_0)$ with $\lVert U_1\rVert\,\lVert V_1\rVert\le\frac{(1+\epsilon)^2}{(1-\epsilon)\xi_1}$.

Now we split in two cases, beginning with the case where the restriction of~$R_2$ to~$W_1$ is non-compact. Since~$W_1$ is isomorphic to~$c_0$, we can apply \Cref{L:factorizationdiagram} once more, this time to the operator $R_2|_{W_1}$, to find a constant $\xi_2>0$ and a closed, in\-fi\-nite-di\-men\-sional subspace~$W_2$ of~$W_1$ such that
\begin{equation}\label{JPSsubmult:eq7}
      (1+\epsilon)\xi_2\lVert w\rVert\ge \lVert R_2w\rVert\ge (1-\epsilon)\xi_2\lVert w\rVert\qquad (w\in W_2)
\end{equation}    
and $I_{c_0} = U_2R_2V_2$ for  some operators $U_2,V_2\in\mathscr{B}(c_0)$ with $\lVert U_2\rVert\,\lVert V_2\rVert\le\frac{(1+\epsilon)^2}{(1-\epsilon)\xi_2}$. 

Then, for $i\in\{1,2\}$, we have $I_{c_0} = U_iR_iV_i = (U_iU)T_i(VV_i)$, so 
\begin{align}\label{JPSsubmult:eq8}
  \nu(T_i+\mathscr{K}(C_0(K)))&\ge \frac1{\lVert U_i\rVert\,\lVert U\rVert\,\lVert V\rVert\,\lVert V_i\rVert}\ge\frac{(1-\epsilon)^2\xi_i}{(1+\epsilon)^2}\nu(T_1+T_2+\mathscr{K}(C_0(K))).   
\end{align} 
The left-hand inequalities in~\eqref{JPSsubmult:eq6}--\eqref{JPSsubmult:eq7} and the fact that $W_2\subseteq W_1$ imply that
\begin{equation}\label{JPSsubmult:eq9}
    (1+\epsilon)(\xi_1+\xi_2)\ge \lVert (R_1+R_2)|_{W_2}\rVert\ge \lVert (R_1+R_2)|_{W_2}\rVert_e = \lVert I_{c_0}|_{W_2}\rVert_e =1, 
\end{equation} 
where the penultimate equality follows from~\eqref{JPSsubmult:eq4}. 
Adding up the estimates~\eqref{JPSsubmult:eq8} for $i=1$ and $i=2$ and substituting the lower bound on~$\xi_1+\xi_2$ from ~\eqref{JPSsubmult:eq9} into this sum, we obtain
\begin{equation}\label{JPSsubmult:eq10} \nu(T_1+\mathscr{K}(C_0(K))) + \nu(T_2+\mathscr{K}(C_0(K)))\ge \frac{(1-\epsilon)^2}{(1+\epsilon)^3}\nu(T_1+T_2+\mathscr{K}(C_0(K))).   \end{equation}

We claim that this inequality also holds true in the case where~$R_2|_{W_1}$ is compact. Indeed, the estimate~\eqref{JPSsubmult:eq8} remains valid for $i=1$, while we can modify~\eqref{JPSsubmult:eq9} in the  following way:
\[ (1+\epsilon)\xi_1\ge \lVert R_1|_{W_1}\rVert\ge \lVert R_1|_{W_1}\rVert_e = \lVert (R_1+R_2)|_{W_1}\rVert_e = \lVert I_{c_0}|_{W_1}\rVert_e =1.
\]
Now~\eqref{JPSsubmult:eq10} follows by using the trivial lower bound~$0$ on~$\nu(T_2+\mathscr{K}(C_0(K)))$.

Since~\eqref{JPSsubmult:eq10} holds true for arbitrary $\epsilon\in(0,1)$, we conclude that~$\nu$ is subadditive.\medskip

\textbf{Submultiplicativity.} The inequality \[ \nu(T_2T_1+\mathscr{K}(C_0(K)))\le\nu(T_2+\mathscr{K}(C_0(K)))\nu(T_1+\mathscr{K}(C_0(K))) \] is trivial if~$T_2T_1$ is compact.
Otherwise, for each $\epsilon\in(0,1)$, we can find operators $U\in\mathscr{B}(C_0(K),c_0)$ and $V\in\mathscr{B}(c_0,C_0(K))$ such that 
\begin{equation}\label{JPSsubmult:eq1}
 I_{c_0} - UT_2T_1V\in\mathscr{K}(c_0)\qquad\text{and}\qquad   
\frac1{\lVert U\rVert\,\lVert V\rVert}\ge (1-\epsilon)\nu(T_2T_1+\mathscr{K}(C_0(K))). 
\end{equation}
Applying \Cref{L:factorizationdiagram} and \Cref{L:HNA} twice, we obtain constants $\xi_1,\xi_2>0$, closed sub\-spaces $W_1\subseteq V[c_0]$ and $W_2\subseteq T_1[W_1]$ and operators $U_1,U_2\in\mathscr{B}(C_0(K),c_0)$ and $V_1,V_2\in\mathscr{B}(c_0,C_0(K))$ such that~$T_1[W_1]$ and~$T_2[W_2]$ are isomorphic to~$c_0$ and 
\begin{equation}\label{JPSsubmult:eq2}
   (1+\epsilon)\xi_i\lVert w\rVert\ge \lVert T_iw\rVert\ge (1-\epsilon)\xi_i\lVert w\rVert,\quad U_iT_iV_i = I_{c_0},\quad
   \lVert U_i\rVert\,\lVert V_i\rVert\le\frac{(1+\epsilon)^2}{(1-\epsilon)\xi_i}
\end{equation}    
for $w\in W_i$ and $i\in\{1,2\}$. (To explain this construction in detail, for $i=1$ we apply \Cref{L:factorizationdiagram} to the operator~$T_1|_{V[c_0]}$; this is justified because \Cref{L:HNA} and the first part of~\eqref{JPSsubmult:eq1} show that the domain~$V[c_0]$ is isomorphic to~$c_0$ and~$T_1|_{V[c_0]}$ is non-compact. Next, for $i=2$, we apply \Cref{L:factorizationdiagram} to the operator~$T_2|_{T_1[W_1]}$; by construction its domain~$T_1[W_1]$ is isomorphic to~$c_0$, and since $T_1[W_1]\subseteq T_1V[c_0]$, we can apply \Cref{L:HNA} and~\eqref{JPSsubmult:eq1} once more to deduce that~$T_2|_{T_1[W_1]}$  is non-compact.)  

The second and third identity in~\eqref{JPSsubmult:eq2} show that
\begin{equation}\label{JPSsubmult:eq3}
    \nu(T_2+\mathscr{K}(C_0(K)))\nu(T_1+\mathscr{K}(C_0(K)))\ge \frac1{\lVert U_2\rVert\,\lVert V_2\rVert\,\lVert U_1\rVert\,\lVert V_1\rVert}\ge \frac{(1-\epsilon)^2\xi_1\xi_2}{(1+\epsilon)^4}. 
\end{equation} 
Set $W_0 = V^{-1}[T_1^{-1}[W_2]\cap W_1]$. This is a closed subspace of~$c_0$ such that \[ V[W_0] = T_1^{-1}[W_2]\cap W_1\qquad\text{and}\qquad T_1V[W_0]=W_2. \] The latter identity implies that~$W_0$ is infinite-dimensional. Combining this with the first part of~\eqref{JPSsubmult:eq1} and the left-hand inequality in~\eqref{JPSsubmult:eq2}, we find
\[ 1= \lVert I_{c_0}|_{W_0}\rVert_e = \lVert UT_2T_1V|_{W_0}\rVert_e\le \lVert U\rVert\,\lVert T_2|_{W_2}\rVert\,\lVert T_1|_{W_1}\rVert\,\lVert V\rVert\le (1+\epsilon)^2 \xi_1\xi_2\lVert U\rVert\,\lVert V\rVert. \]
This produces a lower bound on~$\xi_1\xi_2$, which we can substitute into~\eqref{JPSsubmult:eq3} to obtain
\begin{align*} \nu(T_2+\mathscr{K}(C_0(K)))\nu(T_1+\mathscr{K}(C_0(K))) &\ge \frac{(1-\epsilon)^2}{(1+\epsilon)^6\lVert U\rVert\,\lVert V\rVert}\\ &\ge \frac{(1-\epsilon)^3}{(1+\epsilon)^6}\nu(T_2T_1+\mathscr{K}(C_0(K))).
\end{align*}
Since this inequality holds true for arbitrary $\epsilon\in(0,1)$, we conclude that~$\nu$ is submultiplicative.

Finally, the inequality~\eqref{P:JPSalgnorm:eq1} follows from the fact that if $I_{c_0}-UTV\in\mathscr{K}(c_0)$, then 
\[ 1 = \lVert I_{c_0}\rVert_e = \lVert UTV\rVert_e\le \lVert U\rVert\,\lVert T\rVert_e\,\lVert V\rVert. \qedhere \]
\end{proof}

\begin{proof}[Proof of Theorem~{\normalfont{\ref{charCKincompress}}}] We can apply \Cref{P:AequivB} because the unit ball of~$C_0(K)^*$ is weak* sequentially compact by \Cref{weakstardualball:formerpart1}; it  shows that conditions~\ref{charCKincompress1} and~\ref{charCKincompress2} are equivalent.

  We shall now complete the proof by verifying that conditions~\ref{charCKincompress2}--\ref{charCKincompress5} are equivalent.  \Cref{Factorisation} shows that~\ref{charCKincompress2} implies~\ref{charCKincompress3}, which trivially implies~\ref{charCKincompress4}.

  \ref{charCKincompress4}$\Rightarrow$\ref{charCKincompress5}. Suppose that~$\mathscr{B}(C_0(K))/\mathscr{K}(C_0(K))$ is incompressible, and consider the identity map~$\iota$ on~$\mathscr{B}(C_0(K))/\mathscr{K}(C_0(K))$, which is obviously an algebra isomorphism. \Cref{P:JPSalgnorm} shows that~$\iota$ is continuous if we endow its domain with the essential norm and its codomain with the algebra norm~$\nu$ given by~\eqref{eq:def_nu}. Therefore the hypothesis implies that~$\iota$ is bounded below; that is, we can find a constant $\eta>0$ such that
  \begin{equation}\label{eq:20july23} \nu(T+\mathscr{K}(C_0(K)))\ge\eta\lVert T\rVert_e\qquad (T\in\mathscr{B}(C_0(K))). \end{equation}

  \ref{charCKincompress5}$\Rightarrow$\ref{charCKincompress2}. Suppose that $\eta>0$ is  a constant for which~\eqref{eq:20july23} is satisfied, and  note that $\eta\le1$ by~\eqref{P:JPSalgnorm:eq1}.
We claim that~\ref{charCKincompress2} is satisfied for any constant $C>1/\eta$. Indeed, let $T\in\mathscr{B}(C_0(K))$ be a non-compact operator  with $\lVert T\rVert_e =1$, and take $\xi\in(1/C,\eta)$. Then we have  $\nu(T+\mathscr{K}(C_0(K)))\ge\eta>\xi$, so the definition~\eqref{eq:def_nu} of~$\nu$ implies that there are operators $U_1\in\mathscr{B}(C_0(K),c_0)$ and $V_1\in\mathscr{B}(c_0,C_0(K))$ such that 
\[ I_{c_0} -U_1TV_1\in\mathscr{K}(c_0)\qquad\text{and}\qquad \frac{1}{\lVert U_1\rVert\,\lVert V_1\rVert} > \xi. \]
Since $C\xi>1$, \Cref{L:perturb} shows that $U_2(U_1TV_1)V_2 = I_{c_0}$ for some operators $U_2,V_2\in\mathscr{B}(c_0)$ with $\lVert U_2\rVert\,\lVert V_2\rVert< C\xi$. Hence the operators $U= U_2U_1$ and $V=V_1V_2$ have the required properties.
\end{proof}

\subsection*{Acknowledgements} This paper is part of the first-named author's PhD at Lan\-caster Uni\-ver\-sity. He acknowledges with thanks the funding from the EPSRC (grant number EP/R513076/1) that has supported his studies.
 
We are grateful to Bill Johnson (Texas A\&{}M) for sharing an early draft of the manuscript~\cite{JKS}  with us, and to Hans-Olav Tylli (Helsinki) for pointing out \Cref{GST:Thm2.6}.

\end{document}